\documentclass[12pt, a4paper]{article}
\usepackage{graphicx}
\usepackage{tikz}
\usepackage{color}
\usepackage[total={6in,9.5in},a4paper]{geometry}
\usepackage{amsmath, amsthm, amssymb}
\newcommand{\half}{\tfrac{1}{2}}
\newcommand\eqdef{\mathrel{\overset{\makebox[0pt]{\mbox{\tiny def}}}{=}}}

\usepackage{hyperref}

\newtheorem{thm}{Theorem}[section]
\newtheorem*{thm*}{Theorem}
\newtheorem{lem}[thm]{Lemma}
\newtheorem{cor}[thm]{Corollary}
\newtheorem{prop}[thm]{Proposition}
\newtheorem{question}{Question}
  
\theoremstyle{definition}
\newtheorem{definition}{Definition}

\makeatletter
\theoremstyle{remark}
\newtheorem*{remark}{Remark}

\def\beq#1\eeq{\begin{equation}#1\end{equation}}
\@ifpackageloaded{euler}{
 
 \newcommand{\onto}{\to\mkern-14mu\to}
 \def\rightarrowfill@#1{\m@th\setboxz@h{$#1\relbar$}\ht\z@\z@
   $#1\copy\z@\mkern-6mu\cleaders
   \hbox{$#1\mkern-2mu\box\z@\mkern-2mu$}\hfill
   \mkern-6mu\mathord\rightarrow$}
 \def\leftarrowfill@#1{\m@th\setboxz@h{$#1\relbar$}\ht\z@\z@
   $#1\mathord\leftarrow\mkern-6mu\cleaders
   \hbox{$#1\mkern-2mu\copy\z@\mkern-2mu$}\hfill
   \mkern-6mu\box\z@$}

}{
 
 \DeclareMathSymbol{\onto}{\mathrel}{AMSa}{"10}

}

\newcommand{\R}{\mathbb{R}}
\newcommand{\Z}{\mathbb{Z}} 
\newcommand{\N}{\mathbb{N}}

\newcommand{\Norm}[1]{\left\Vert#1\right\Vert}
\newcommand{\abs}[1]{\lvert#1\rvert}
\newcommand{\weaks}{weak${}^*$}

\newcommand{\Or}{\mathcal{O}}
\newcommand{\Abs}[1]{\left|#1\right|}
\DeclareMathOperator{\sgn}{sgn}
\DeclareMathOperator{\card}{card}

\begin{document}

\title{Switch Functions}

\author{Richard R.~Hall, Eli Hawkins and Bernard S.~Kay \medskip \\
{\small \emph{Department of Mathematics, University of York, York YO10 5DD, UK}}\\
{\footnotesize\tt RichardRoxbyHall@gmail.com \ \ eli.hawkins@york.ac.uk \ \ bernard.kay@york.ac.uk}}

\date{}

\maketitle

\begin{abstract}
We define a switch function to be a function from an interval to $\{1,-1\}$ with a finite number of sign changes. (Special cases are the Walsh functions.)  By a topological argument, we prove that, given $n$ real-valued functions, $f_1, \dots, f_n$, in $L^1[0,1]$, there exists a  switch function, $\sigma$, with at most $n$ sign changes that is simultaneously orthogonal to all of them in the sense that $\int_0^1 \sigma(t)f_i(t)dt=0$, for all $i = 1, \dots , n$.   

Moreover, we prove that, for each $\lambda \in (-1,1)$, there exists a unique switch function, $\sigma$, with $n$ switches such that  $\int_0^1 \sigma(t) p(t) dt = \lambda \int_0^1 p(t)dt$ for every real polynomial $p$ of degree at most $n-1$. We also prove the same statement holds for every real even polynomial of degree at most $2n-2$.  Furthermore, for each of these latter results, we write down, in terms of $\lambda$ and $n$, a  degree $n$ polynomial whose roots are the switch points of $\sigma$; we are thereby able to compute these switch functions.
\end{abstract}

\section{Introduction}
In this paper, we provide a positive answer to the first (existence) part of a question raised by one of us  \cite[end of Sec.~4]{HallTrigEllipt}  and also an answer to the second (computation) part of the question in some special cases.  

\begin{definition}
Given a real interval $[a,b]$, we call $\sigma:[a,b]\to \{1, -1\}$ a \emph{switch function}.    We call the points at which it changes sign its \emph{switch points} and sometimes refer to its sign changes as its \emph{switches}.
\end{definition}   
The functions studied by Walsh in \cite{Walsh}, are examples of switch functions.

\begin{definition}
Given  functions $f,\sigma:[a,b]\to\R$, we define the pairing
\begin{equation}
\label{InnProd}\langle f,\sigma\rangle = \int_a^b \sigma(t)f(t)\,dt
\end{equation}
and we say that $\sigma$ and $f$ are \emph{orthogonal} if $\langle f,\sigma\rangle=0$. 
\end{definition}
 The following questions were posed in \cite[p.~543]{HallTrigEllipt}: Under what circumstances does a switch function  with at most $n$ switches exist that is orthogonal to $n$ given functions, and how can such a function be computed?

To begin, we consider  what class of functions to pair with switch functions. The most obvious choice would be continuous functions, but we can do a little better than that. We need functions for which the inner products with switch functions are well defined. For such a function, $f$, there exists a switch function, $\sigma$ (essentially $\sgn f$) such that $f\sigma = \abs f$, and hence
\[
\langle f,\sigma\rangle = \int_a^b \abs{f(t)}dt ,
\]
which is the $L^1$ norm of $f$, is finite. Indeed, this is the supremum of inner products of $f$ with switch functions. So, we should require $f\in L^1[a,b]$. (Functions will always be $\R$-valued, unless otherwise specified.)

Because switch functions are paired with $L^1$ functions, it is natural to treat the switch functions as a subset of the predual of $L^1[a,b]$, which is $L^\infty[a,b]$ with the \weaks\ topology determined by the pairing with $L^1[a,b]$.
\begin{definition}
The \weaks\ topology on $L^\infty[a,b]$ is the weakest topology such that the pairing with any $f\in L^1[a,b]$ is a continuous map, $\langle f,\;\cdot\;\rangle : L^\infty[a,b]\to\R$.
\end{definition}

Note that $L^\infty[a,b]$ isn't really a set of functions. Its elements are equivalence classes of functions that differ on sets of measure $0$. This is good, because we really don't care whether a switch function is equal to $1$ or $-1$ at a given switch point.

Without loss of generality, we can take the domain interval to be $[0,1]$. 

Note that a switch function $\sigma\in L^\infty[0,1]$ is determined, up to an overall sign by its switch points. This parametrization will be very important here.
\begin{definition}
Denote $\mathbf x = (x_1,\dots,x_n)\in\R^n$.
For any $n\in\N$, the \emph{standard $n$-simplex} is
\[
\Delta^n \eqdef \{\mathbf x\in\R^n \mid 0\leq x_1\leq x_2\leq\dots\leq x_n\leq 1\} .
\]
For convenience, we let $x_0\eqdef0$ and $x_{n+1}\eqdef1$.
\[
\partial\Delta^n \eqdef \{\mathbf x\in\Delta^n \mid x_m=x_{m+1} \text{ for some }0\leq m\leq n\}
\]
\end{definition}

\begin{definition}
For any $\mathbf x\in\Delta^n$, let $\Sigma_n(\mathbf x)\in L^\infty[0,1]$ be (the equivalence class of) the function given by
\[
\Sigma_n(\mathbf x)(t) = (-1)^m \quad \text{ for }\quad x_m \leq t < x_{m+1}
\]
for any $0\leq m\leq n$, and $\Sigma_n(\mathbf x)(1)=(-1)^n$.
In this way, $\Sigma_n : \Delta^n \to L^\infty[0,1]$.
\end{definition}

\begin{definition}
$D_n \eqdef \Sigma_n(\Delta^n) \subset L^\infty[0,1]$  with the \weaks\ topology. $\partial D_n \eqdef \Sigma_n(\partial \Delta^n) \subset D_n$.
\end{definition}
\begin{remark}
$D_n\smallsetminus \partial D_n \subset L^\infty[0,1]$ is the set of switch functions with precisely $n$ switches that take the value $+1$ for all small enough $t\in [0,1]$. $\partial D_n =D_{n-1}\cup (-D_{n-1})\subset L^\infty[0,1]$ is the set of switch functions with at most $n-1$ switches.
\end{remark}

With these definitions, we can formulate the main questions more precisely.
\begin{question}
\label{Main Question}
Given $n$ functions, $f_1,\dots,f_n\in L^\infty[0,1]$ (or equivalently, an $n$-dimensional subspace) does there exist $\sigma\in D_n$ such that $\langle f_i,\sigma\rangle=0$ for $i=1,\dots,n$? Can such a $\sigma$ be computed? Is it unique?
\end{question}
Theorem~\ref{Main theorem} will show that $\sigma$ does always exist. In the later sections, we will compute it for some classes of functions.

Given a single function, $f\in L^1[0,1]$, it is clear that there exists $\sigma\in D_1$ orthogonal to $f$,  because if we set
\begin{equation}
\label{IndInt}
F(x)=\int_0^x f(t)dt
\end{equation}
then $F$ is continuous and so there exists \emph{at least} one $x_1$ such that $F(x_1)={\half}F(1)$. This serves as the switch point, so $\sigma=\Sigma_1(x_1)$ (and $-\Sigma_1(x_1)$) is a switch function orthogonal to $f$.

This is a special case of the Intermediate Value Theorem, so Theorem~\ref{Main theorem} can be seen as a generalization of the Intermediate Value Theorem.

Given $n$ functions, $f_i$, $i= 1, \dots, n$, with integrals $F_i$, $\Sigma_n(\mathbf x)$ is an orthogonal switch function if $0\le x_1\le x_2 \le \dots \le x_n \le 1$ and
\begin{equation}
\label{SwitchEq}
F_i(x_n)-F_i(x_{n-1}) + \dots + (-1)^{n-1}F_i(x_1)=\half F_i(1),
\end{equation}
for $1\le i\le n$.

\section{Two Switches}
When $n=2$, the equations \eqref{SwitchEq} specialise to
\begin{equation*}
F_i(x_2)-F_i(x_1)=\half F_i(1),
\end{equation*}
for $i=1,2$.   There is an easy solution in this case when one of the functions, say $f_1$, is positive (i.e.~$f_1(x) > 0$ for almost all $x\in [0,1]$).  By rescaling, we can assume without loss of generality that  $F_1(1)=1$.  
Define $a$ by $F_1(a)=\half$ and $\phi:[0,a]\to\R$  by the equation
\begin{equation}
\label{f1pos}
F_1(x+\phi(x))= F_1(x) + \half .
\end{equation}
Note that $a+\phi(a)=1$.  If $F_2(a)$ is $\half F_2(1)$ then we only need one switch point.  Otherwise, the equation
\begin{equation}
\label{ExcessDeficit}
F_2(x+\phi(x)) - F_2(x) = \half F_2(1)
\end{equation}
errs by excess or deficit at $x=0$ and the opposite at $x=a$, and therefore, by continuity, there must be a value, $x=x_1$ for which it holds, whereupon we may take the switch points to be $x_1$ and $x_2=x_1+\phi(x_1)$.   (When $F_2(a)$ is $\half F_2(1)$, we may, if we wish, think that there is a second switch point at $x=0$ or at $x=1$.)

This argument permits an extension.  Given an integer $k \ge 3$ we can define $\phi$ by $F_1(x+\phi(x))= F_1(x) + \frac1k$, together with points $a_1, a_2, \dots, a_{k-1}$ ($a_1=\phi(0), a_2=a_1 + \phi(a_1), \dots , a_{k-1}=a_{k-2} + \phi(a_{k-2}$)) such that $F_1(a_j)=j/k$, ($1\le j < k$).  We put $a_0=0$, $a_k=1$ ($=a_{k-1} +\phi(a_{k-1})$) and look at the integrals of $f_2$ over the intervals $[a_{j-1}, a_j]$.  Either these are all equal to $F_2(1)/k$ (in which case \eqref{kExDef} below holds with $x$ equal to any of the $a_j$, $j=1, \dots , {k-1}$)  or there exists an adjacent pair of such integrals, one in excess and the other deficit, and so there exists an $x$ such that 
\begin{equation}
\label{kExDef}
F_2(x+\phi(x)) - F_2(x) = \frac{1}{k} F_2(1)
\end{equation}
Thus we may solve the simultaneous equations
\begin{equation}
\label{1overk}
F_i(x_2)-F_i(x_1)=\frac{1}{k} F_i(1), \quad i=1,2.
\end{equation}
We then have that for any $\lambda$ of form
\begin{equation}
\label{lambda}
\lambda=1-\tfrac{2}{k},
\end{equation}
where $k\in\N$, and for any pair of functions $f_1, f_2 \in L^1[0,1]$, one of which is positive, there exists a $0\leq x_1\leq x_2\leq 1$ such that $\sigma=\Sigma_2(x_1,x_2)$ satisfies
\begin{equation}
\label{ChiLambda}
\int_0^1 \sigma(x)f_i(x)dx= \lambda\int_0^1f_i(x)dx,
\end{equation}
for $i=1,2$. In other words, $\sigma-\lambda$ is orthogonal to each $f_i$.

This suggests a more general question.
\begin{question}
\label{Generalized Question}
Given $-1\leq\lambda\leq1$ and $n$ functions, $f_1,\dots,f_n\in L^1[0,1]$, does there exist $\sigma\in D_n$, such that $\langle f_i,\sigma-\lambda\rangle=0$ for $i=1,\dots n$? Can such a $\sigma$ be computed? Is it unique?
\end{question}

For now, we return to Question~\ref{Main Question} (i.e., $\lambda=0$) for $n=2$.   We can drop the assumption that one of the two functions, $f_1, f_2$, is positive with the following topological argument which, in the sequel, we will generalize to the case of $n$ functions.  

Recall that the simplex $\Delta^2$ is just the triangle in $\R^2$ with vertices $(0,0)$, $(0,1)$, and $(1,1)$; 
\begin{figure}
   \centering
    \includegraphics[clip, scale=0.7]{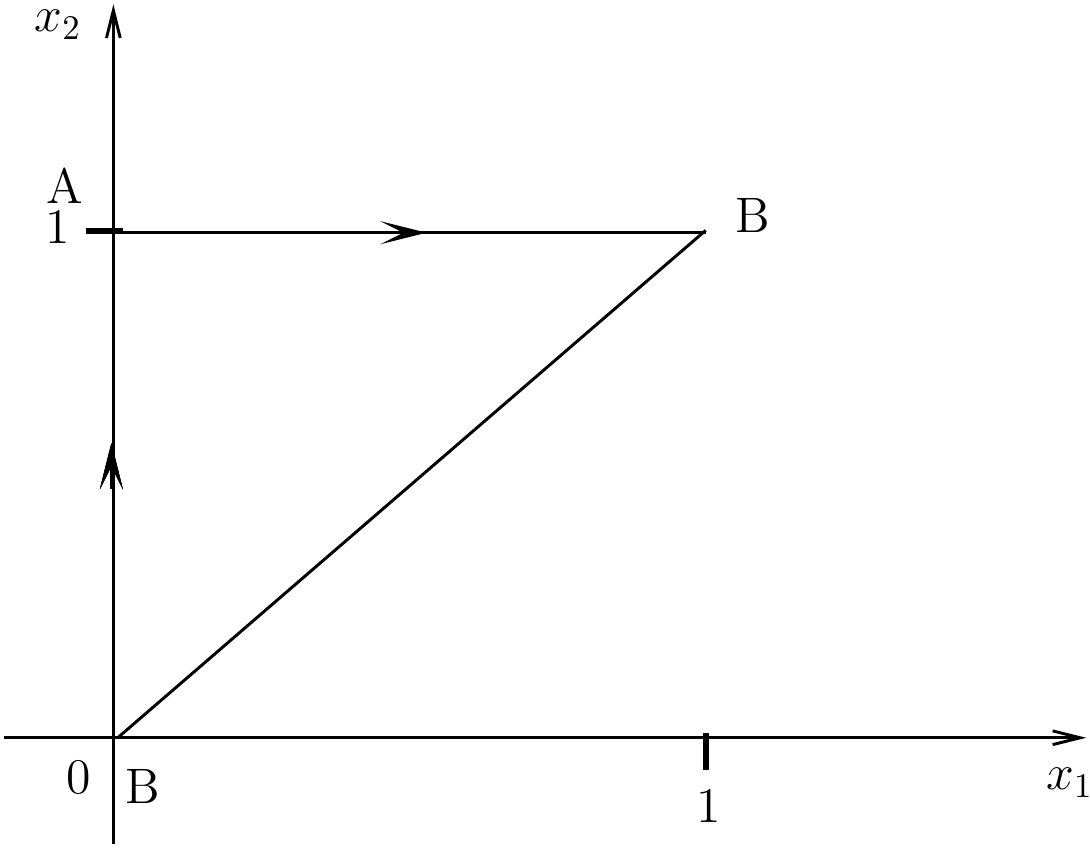}
   \caption{ \label{fig1} The triangle $\Delta^2$, which, as explained in the text, can be taken to represesent the set of switch functions $D_2$ provided one thinks of points on the hypoteneuse as identified.  For the significance of the labels A and B and of the arrows, see the caption to Figure~\ref{fig2}.}
 \end{figure}
this is homeomorphic to the closed unit disc, $B^2$. The map $\Sigma_2:\Delta^2\to D_2$ is almost injective. Regarded as an element of $L^\infty[0,1]$, a function that switches twice at the same point is the same as a function that doesn't switch at all, so 
\begin{equation*}
\Sigma_2(x,x)=1 \quad \forall x\in [0,1],
\end{equation*}
but any other switch function in $D_2$ comes from a unique point in $\Delta^2$. This means that $D^2$ is topologically the quotient of $\Delta^2$ by its hypotenuse, the edge from $(0,0)$ to $(1,1)$. This is also homomorphic to $B^2$.

The upper edge of $\Delta^2$ is the set $\{(x,1)\mid x\in[0,1]\}$. For $0<x<1$, $\Sigma_2(x,1)$ is the switch function with a single switch point at $x$ and equal to $+1$ on $[0,x)$. At the ends, this reduces to the constant functions $\Sigma_2(0,1)=-1$ and $\Sigma_2(1,1)=+1$ with no switch points.

The vertical edge of $\Delta^2$ is the set $\{(0,x)\mid x\in[0,1]\}$. For $0<x<1$, $\Sigma_2(0,x)$ is the switch function with a single switch point at $x$ but equal to $-1$ on $[0,x)$.

When $\mathbf x\in\Delta^2$ is in the interior, $\Sigma_2(\mathbf x)$ is the unique switch function in $D_2$ with switch points $x_1$ and $x_2$. When $\mathbf x\in\Delta^2$ is on the boundary, $\Sigma_2(\mathbf x)$ has one (or no) switch point, and there is another function in $D_2$ with the same switch point. To be precise, $\Sigma_2(0,x)=-\Sigma_2(x,1)$. 

Analogously, any point $\mathbf z\in S^1\subset B^2$, on the unit circle, which is the boundary of the unit disc, has an antipodal point $-\mathbf z\in S^1$. In fact, we will see that there exists a (not unique) homeomorphism $\varphi_2:B^2\to D_2$ such that for $\mathbf z\in S^1$, $\varphi_2(-\mathbf z) = -\varphi_2(\mathbf z)$.

For two given functions, $f_1, f_2\in L^1[0,1]$, we now consider the map, $\Psi: D_2 \rightarrow \mathbb{R}^2$
\begin{equation}
\label{Phi} 
\Psi(\sigma) \eqdef \left(\int_0^1 \sigma(t)f_1(t)dt, \int_0^1\sigma(t)f_2(t)dt\right) = \left(\langle f_1,\sigma\rangle,\langle f_2,\sigma\rangle\right).
\end{equation}
This is continuous, by definition of the \weaks\ topology, and for $\sigma\in\partial D_2$, $\Psi(-\sigma)=-\Psi(\sigma)$.   Precomposing $\Psi$ with the homeomorphism, $\varphi_2:B^2\to D_2$ we thus have a continuous map $\Psi\circ \varphi_2:B^2\to\mathbb{R}^2$ that maps antipodal points to antipodal points. 

Two simple examples are provided by (a) $f_1(x)=1, f_2(x)=x$ and (b) $f_1(x)=\cos \pi x, f_2(x)=\sin \pi x$.   For both of these examples, it is not difficult to compute the map $\Psi$ and to verify that  $\Psi$ is a homeomorphism onto its image. (See Lem.~\ref{Hom lemma}.)  

In fact, as indicated in Figure~\ref{fig2}(a), in the first example, the range is the region bounded by the curves $x \mapsto (y_1(x), y_2(x))$ and $x\mapsto (-y_1(x), -y_2(x))$ where
\begin{equation*}
(y_1(x), y_2(x)) = \Psi(\Sigma_2(0,x))  = (1-2x, \half - x^2)
\end{equation*}
is the image of the point $(0,x)$ on the vertical edge of the triangle, $\Delta^2$, pictured in Figure~\ref{fig1}.   

Similarly, in our second example, the range of $\Psi$ is the disc bounded by the circle of radius $2/\pi$ centred on the origin. The image of the point $(0,x)$ along the vertical edge of $\Delta^2$ is 
\begin{equation*}
(y_1(x), y_2(x)) = \Psi(\Sigma_2(0,x))  = \frac{2}{\pi}(-\sin \pi x, \cos \pi x)
\end{equation*}
on the left semicircle. Likewise, the image of the point $(x,1)$ on the horizontal edge of $\Delta^2$ is  $(-y_1(x), -y_2(x))$ on the right semicircle. 

\begin{lem}
\label{Hom lemma}
There exists a homeomorphism $\varphi_2:B^2\to D_2$ such that for $\mathbf z\in S^1$, $\varphi_2(-\mathbf z)=-\varphi_2(\mathbf z)$.
\end{lem}
\begin{proof}
For case (b) above, the Jacobian determinant of $\Psi\circ\Sigma_2 : \Delta^2\to\R^2$ is easily computed to be $4\sin[\pi(x_2-x_1)]$, which does not vanish on the interior of $\Delta^2$; therefore $D_2\smallsetminus\partial D_2$ maps homeomorphically to its image. Explicit calculation shows that $\partial D_2$ maps homeomorphically to the circle of radius $\frac{2}{\pi}$. This shows that $\Psi$ maps $D_2$ homeomorphically to the disc of radius $\frac2\pi$. 

After rescaling by $\frac\pi2$, this gives an explicit homeomorphism from $D_2$ to $B^2$. Its inverse is a suitable choice of $\varphi_2$.
\end{proof}
Of course, this is far from unique.

\begin{prop}
\label{Preliminary case}
Given $f_1,f_2\in L^1[0,1]$, there exists $\sigma\in D_2$ orthogonal to $f_1$ and $f_2$.
\end{prop}
\begin{proof}
This condition is equivalent to $\Psi(\sigma)=(0,0)$. Suppose that no such $\sigma$ exists.

Any map from $S^1$ to $\R^2\smallsetminus \{(0,0)\}$ has a well-defined, homotopy-invariant winding number (around the origin).  For $s\in[0,1]$, define $\alpha_s:S^1\to \R^2\smallsetminus \{(0,0)\}$ by
\[
\alpha_s(\mathbf z) = (\Psi\circ \varphi_2)(s\mathbf z) .
\]
Because $\alpha_0$ has image a single point, its winding number is $0$.
Because $\alpha_s$ is continuous in $s$, each $\alpha_s$ must have the same winding number: $0$.

On the other hand, $\alpha_1$ respects antipodes in the sense that $\alpha_1(-\mathbf z)=-\alpha_1(\mathbf z)$ for any $\mathbf z\in S^1$. This implies that $\alpha_1$ has an odd winding number. This is a contradiction.
\end{proof}

This proof used the homeomorphism $\varphi_2:B^2\to D_2$.   For general $n$, it is clear that $\Delta^n\cong B^n$, and $\Sigma_n:\Delta^n\to D_n$ is a homeomorphism on the interior, so that $D_n$ can be constructed as a quotient of $\Delta^n$ by identifications along the boundary.  But, although it seems to us likely that $D_n$ is homeomorphic to the $n$-ball $B^n$, it seems to become increasingly difficult to prove this as $n$ increases and we have not been able to prove it for general $n$.

However, it turns out that homeomorphism is not really necessary for this proof. The proof of Proposition~\ref{Preliminary case} uses algebraic topology, and algebraic topology invariants are not just homeomorphism invariant, they are homotopy invariant \cite{AlgTop}. So, we really only needed $\varphi_2$ to be a homotopy equivalence $\varphi_2:(B^2,S^1)\to(D_2,\partial D_2)$ that respects antipodes. 
Generalizing this will lead to our main existence result, Theorem \ref{Main theorem}.

Returning to Question~\ref{Generalized Question}, despite the above result for $\lambda$ of form $1-2/k$, existence does not always hold.  In fact, if we make the definition:
\begin{definition}
Let $\mathcal F_n$ be the set of $\lambda$ in the interval $(-1,1)$ such that for some $f_1,\dots,f_n\in L^1[0,1]$ there \textbf{does not} exist $\sigma\in D_n$ with $\langle f_i,\sigma-\lambda\rangle = 0$ for all $i=1,\dots,n$.
\end{definition}
then we have
\begin{prop}
\label{Dense}
$\mathcal F_2 \subset [-1,1]$ is dense.
\end{prop}
\begin{proof}
For $k=1,2,3,\dots$ and $m=1,2,\dots,k$,  let 
\[
\lambda = \frac{4k-8m+1}{4k+1} .
\]
Consider the two functions, $f_1(t)=1$ and 
\[
f_2(t) = \frac{(4k+1)\pi}{2} \cos\left(\frac{4k+1}{2}\pi t\right) .
\]
These have been chosen so that $\int_0^1 f_i(t)dt =1$ for $i=1,2$.

The conditions that $0 = \langle f_i,\Sigma_2(x_1,x_2)-\lambda\rangle$ are explicitly
\[
x_2-x_1 = \frac{1-\lambda}2 = \frac{4m}{4k+1}
\]
and 
\begin{align*}
\frac{1-\lambda}2 &= \sin\left(\frac{4k+1}{2}\pi x_2\right) - \sin\left(\frac{4k+1}{2}\pi x_1\right) \\
&= 2 \sin\left(\frac{4k+1}{4}\pi [x_2-x_1]\right) \cos\left(\frac{4k+1}{4}\pi [x_2+x_1]\right) .
\end{align*}
By the first condition,
\[
\sin\left(\frac{4k+1}{4}\pi [x_2-x_1]\right) = \sin \pi m = 0 ,
\]
but $1-\lambda\neq0$. Therefore, for this choice of $\lambda$ and functions, there does not exist any $\sigma\in D_2$ such that $\langle f_i,\sigma-\lambda\rangle =0$ for $i=1,2$, so $\lambda\in \mathcal F_2$.

Clearly, the set 
\[
\left\{\tfrac{4k-8m+1}{4k+1} \mid k,m\in\N,\ 1\leq m\leq k\right\} \subset \mathcal F_2
\]
is dense in the interval $[-1,1]$, and therefore $\mathcal F_2$ is as well.
\end{proof}

\begin{figure}
   \centering
    \includegraphics[clip, scale=0.7]{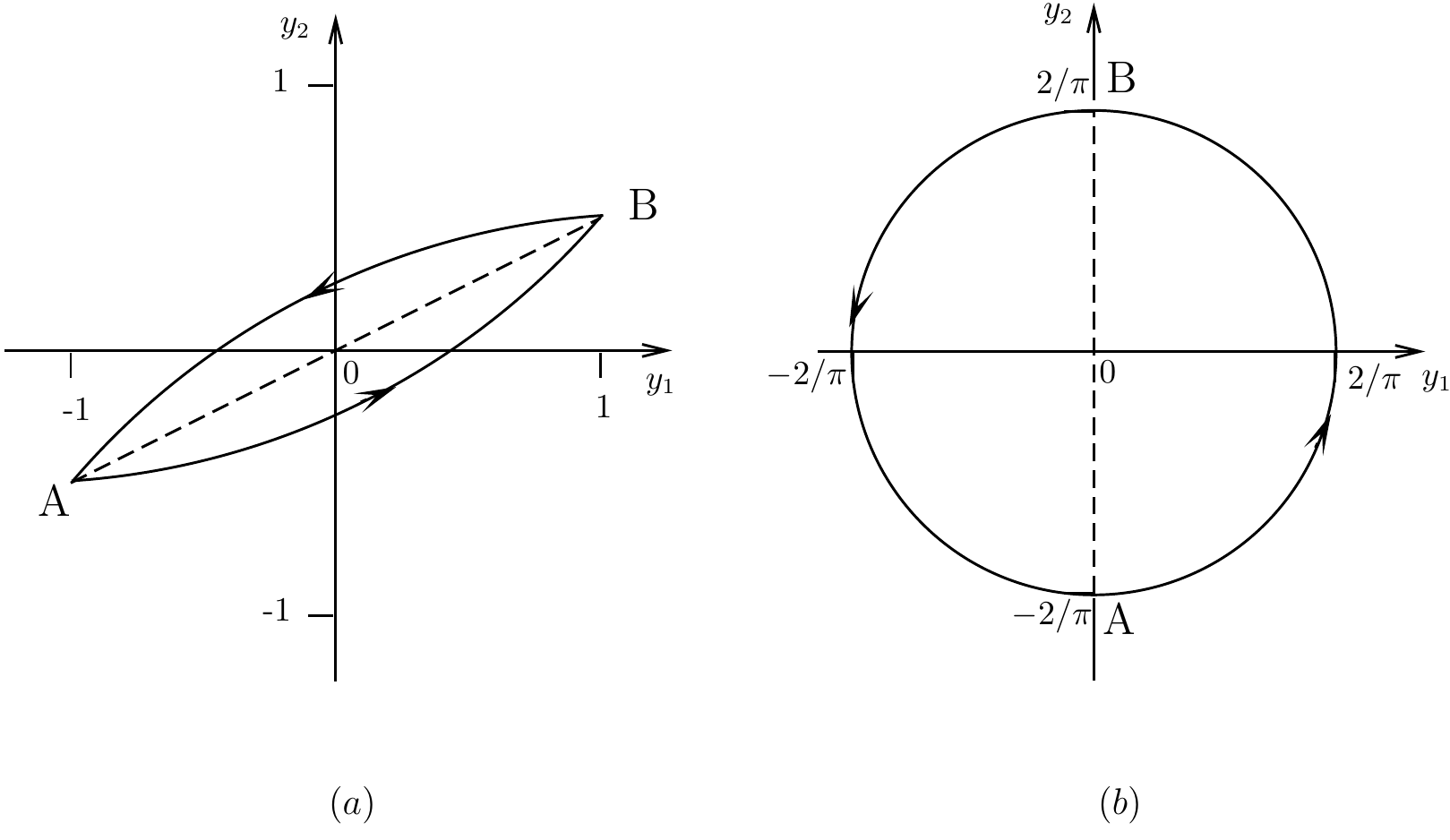}
   \caption{\label{fig2} The image of the map $\Psi: D_2\rightarrow {\mathbb R}^2$ for (a) $f_1(x)=1, f_2(x)=x$; (b) $f_1(x) = \cos \pi x, f_2(x)=\sin \pi x$.   A and B denote the images under $\Psi$ of the points labelled A and B on the triangle, $\Delta^2$ of Figure~\ref{fig1} (regarded as a representation of $D_2$ by identifying the points on its hypotenuse).   The arrows indicate the path traversed by the image of $\Psi$ when the argument traverses the path on the boundary of $\Delta^2$ indicated by the arrows in Figure~\ref{fig1}. The dashed lines in (a) and (b) are the images of the parametrized curves $[-1,1]\rightarrow \mathbb{R}^2$ defined by Equation \eqref{slope}.}
 \end{figure}

On the other hand, existence and uniqueness do hold  in Examples (a) and (b) above in the case $n=2$:  In each of these examples, the line:
\begin{equation}
\label{slope}
\lambda \mapsto \left(\lambda\int_0^1 f_1(t)dt,  \ \ \lambda\int_0^1 f_2(t)dt\right),
\end{equation}
$\lambda \in [-1,1]$,   represented in Figures 2(a) and 2(b) by the indicated dashed lines, clearly lies entirely in the range of $\Psi$.  Hence, for each of these examples, \eqref{ChiLambda} holds for all $\lambda \in [-1,1]$.  In general, if the image of $\Psi$ is star shaped (as in these examples) then existence will hold for all such $\lambda$.

In both of these examples, $\Psi$ is a homeomorphism to its image -- and thus injective. Injectivity implies uniqueness of the switch function.

\begin{figure}
\centering
\includegraphics[scale=.7]{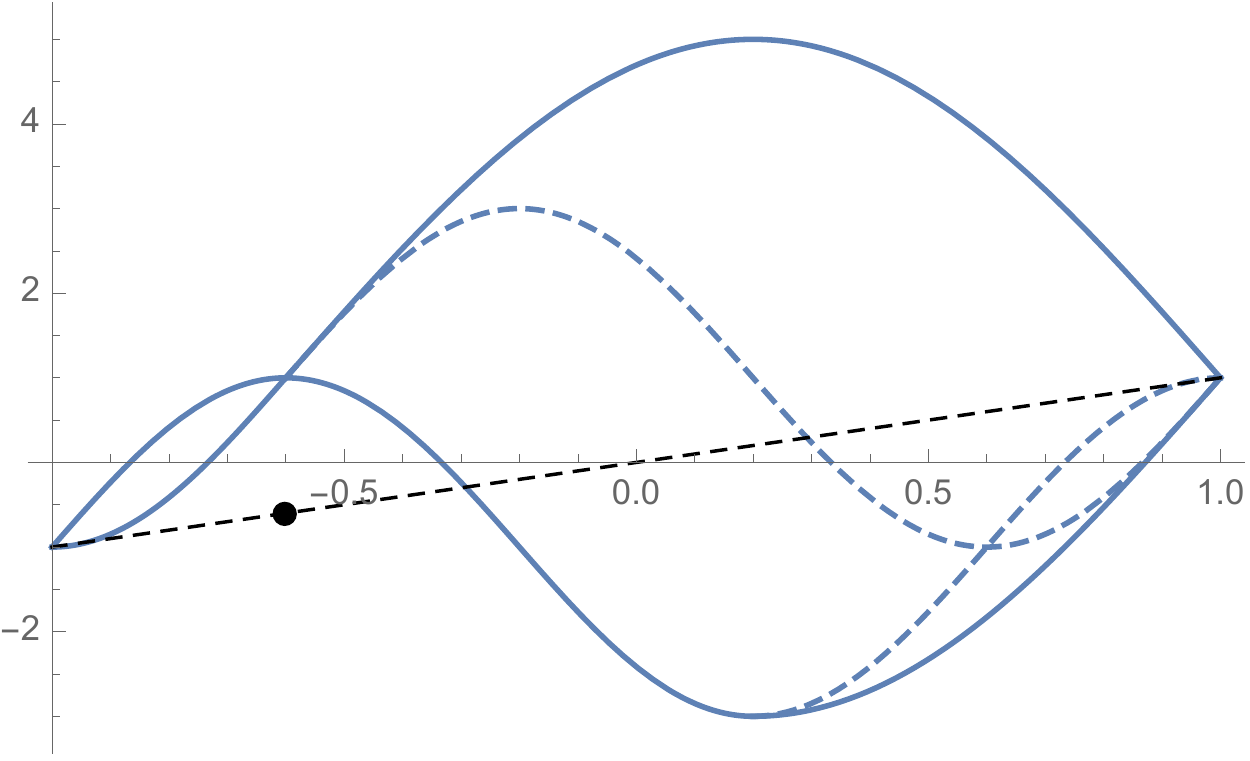}
\caption{The outline of the image of $\Psi$ in the $k=1$ case. The dashed curve is the part of the image of $\partial D_2$ in the interior of the image of $\Psi$. The dashed line segment is parametrized by \eqref{slope}, as in the previous examples, and the marked point is $(-\frac35,-\frac35)$.\label{Counterexample}}
\end{figure}
Figure~\ref{Counterexample} shows the image of $\Psi$ in the $k=1$ case in the proof of Proposition~\ref{Dense}, i.e., $f_1(t)=1$ and 
\[
f_2(t) = \frac{5\pi}{2} \cos\left(\frac{5\pi t}2\right) .
\]
In this case, $\Psi$ is far from injective. It twists and folds $D_2$. The image is not star shaped, and the dashed line segment is not entirely contained in the image. In particular, the point $(-\frac35,-\frac35)$ is outside of the image and this shows that these functions give a counterexample for Question~\ref{Generalized Question} with $\lambda=-\frac35$. Note that in contrast to Figures~\ref{fig2}a and \ref{fig2}b, the image $\Psi(D_2)$ is not invariant under the antipodal map, but in all cases the image of the boundary, $\Psi(\partial D_2)$ is invariant.

Finally, our topological proof of Proposition~\ref{Preliminary case} is no help to actually computing the desired switch function, but we will discuss examples below in which the switch function can be computed explicitly.

\section{An Existence Theorem}
We will now provide a positive answer to the existence part of Question~\ref{Main Question}.
\label{Existence Section}
\begin{lem}
The map $\Sigma_n : \Delta^n\to L^\infty[0,1]$ is \weaks-continuous.
\end{lem}
\begin{proof}
Recall that \weaks\ continuity means precisely that if $\Sigma_n: \Delta^n\to L^\infty[0,1]$ is composed with the pairing with any $f\in L^1[0,1]$, then the result is continuous.

Firstly, observe that 
\[
F(s) \eqdef \int_0^s f(t) dt 
\]
defines a continuous function, $F:[0,1]\to\R$. For any $\mathbf x\in\Delta^n$, the pairing of $\Sigma_n(\mathbf x)$ with $f$ is
\[
\langle f,\Sigma_n(\mathbf x)\rangle = (-1)^n F(1) - \sum_{m=1}^n 2(-1)^m F(x_m) ,
\]
which is manifestly continuous.
\end{proof}

Note that $\Sigma_n$ is injective, except on the subset
\[
K_n \eqdef \{\mathbf x\in\Delta^n \mid x_m=x_{m+1} \text{ for some }1\leq m\leq n-1\} \,.
\]
\begin{lem}
\label{Good subset}
\[
\Sigma_n : \Delta^n\smallsetminus K_n \to D_n\smallsetminus D_{n-2}
\]
is a homeomorphism.
\end{lem}
\begin{proof}
It is easy to see that this is a bijection. We just need to check that the inverse is continuous by showing that the image of an open set is open. It is sufficient to check this for an arbitrarily small neighborhood of any point.

For any $\mathbf y\in\Delta^n$ and $\varepsilon>0$,
\[
\Or_{\mathbf y,\varepsilon} \eqdef \left\{\mathbf x\in \Delta^n \bigm| \left|x_m-y_m\right|<\varepsilon\ \forall m=1,\dots,n\right\}
\]
is an open neighborhood of $\mathbf y$.

Suppose that $\mathbf y \in \Delta^n \smallsetminus \partial\Delta^n$, i.e., $y_m<y_{m+1}$ for $m=0,\dots,n$. If $\varepsilon>0$ and $y_{m+1}-y_m\geq 2\varepsilon$ for $m=0,\dots,n$, then $\Or_{y,\varepsilon} \subset \Delta^n \smallsetminus K_n$ and
\[
\Sigma_n(\Or_{\mathbf y,\varepsilon}) = \bigcap_{m=1}^n \left\{\sigma\in D_n \bigm| \Abs{\left<\chi_{[y_m-\varepsilon,y_m+\varepsilon]},\sigma\right>} < 2\varepsilon\right\} 
\]
where $\chi_{[y_m-\varepsilon,y_m+\varepsilon]}$ is the function equal to $1$ on that interval and $0$ elsewhere. 

By the definition of the \weaks\ topology, the pairing with $\chi_{[y_m-\varepsilon,y_m+\varepsilon]}$ is continuous. The inverse image of the open interval $(-2\varepsilon,2\varepsilon)$ is therefore open, and a finite intersection of open sets is open. 
This shows that $\Sigma_n(\Or_{\mathbf y,\varepsilon})$ is open.

This argument must be modified slightly to extend to any $\mathbf y\in \Delta^n\smallsetminus K_n$. If $y_1=0$ then we use the set 
\[
\left\{\sigma\in D_n \bigm| \Abs{\left<\chi_{[0,\varepsilon]},\sigma\right>}<\varepsilon\right\}
\]
in the first place, and don't require $y_1-y_0\geq 2\varepsilon$.
If $y_n=1$ then we use 
\[
\left\{\sigma\in D_n \bigm| \Abs{\left<\chi_{[1-\varepsilon,1]},\sigma\right>}<\varepsilon\right\} 
\]
in the last place, and don't require $y_{n+1}-y_n\geq 2\varepsilon$.
\end{proof}

\begin{definition}
$B^n\subset \R^n$ is the closed unit ball. $S^{n-1}\subset B^n$ is the unit sphere. 
\end{definition}
\begin{lem}
For all $n\geq0$, there exists a homotopy equivalence $\varphi_n : (B^n,S^{n-1}) \to (D_n,\partial D_n)$, such that for any $\mathbf z\in S^{n-1}$,
\beq
\label{antipode}
\varphi_n(-\mathbf z)=-\varphi_n(\mathbf z) .
\eeq
\end{lem}
\begin{proof}
We will prove this by induction. The base cases $n=0$ and $n=1$ follow immediately if we define the maps by $\varphi_0(0)=+1$ and
\[
\varphi_1(z) = \Sigma_1\left(\tfrac12[z+1]\right) .
\]

Now, let $n\geq2$ and suppose that the proposition is true up to $n-1$. 

Note that $\partial D_n = D_{n-1} \cup (-D_{n-1})$ and the intersection is precisely $\partial D_{n-1}$. Define a map $\psi_n : S^{n-1} \to \partial D_n$ by
\[
\psi_n(\mathbf z,\pm\sqrt{1-\lVert \mathbf z\rVert^2}) = \pm \varphi_{n-1}(\pm \mathbf z)
\]
for $\mathbf z\in B^{n-1}$.  This is well defined, because for $\mathbf z\in S^{n-2}$, eq.~\eqref{antipode} shows that the $+$ and $-$ formulae agree.

The induction hypothesis that $\varphi_{n-1}:(B^n,S^{n-1})\to (D_n,\partial D_n)$ is a homotopy equivalence implies that $\psi_n:S^{n-1}\to \partial D_n$ is a homotopy equivalence.

Note that $K_n\subset \partial \Delta^n$ is the union of a proper subset of closed faces. It is therefore compact and contractable. 
The image $\Sigma_n(K_n)= D_{n-2}$ is compact. By the induction hypothesis, this is homotopic to $B^{n-2}$, so it is also contractable.

By Lemma~\ref{Good subset}, $\Sigma_n$ gives a homeomorphism from the quotient $\partial \Delta^n/K_n$ to $D_n/D_{n-2}$. This gives an isomorphism in  homology $H_{n-1}(\partial \Delta^n/K_n)\to H_{n-1}(D_n/D_{n-2})$. By excision, this is equivalent to an isomorphism $H_{n-1}(\partial \Delta^n,K_n)\to H_{n-1}(D_n,D_{n-2})$. Because $K_n$ and $D_{n-2}$ are contractable (and using the exact sequence) this is equivalent to an isomorphism $H_{n-1}(\partial\Delta^n)\to H_{n-1}(\partial D_n)$.

Now, since $\partial \Delta^n\cong S^{n-1}$ and we have shown that $\partial D_n\simeq S^{n-1}$, this is equivalent to a map between $n-1$ dimensional spheres. Such maps are classified up to homotopy by their action on $H_{n-1}$, therefore $\Sigma_n: \partial\Delta^n\to \partial D_n$ is a homotopy equivalence.

Because $\Sigma_n:\Delta^n\to D_n$ is a homeomorphism on the interior, $D_n$ can be constructed by attaching an $n$-cell to $\partial D_n$ by $\Sigma_n: \partial\Delta^n\to \partial D_n$. That is, $D_n$ is homeomorphic to the quotient of the disjoint union $\Delta^n\cup \partial D_n$ by the equivalence relation generated by $x\sim \Sigma_n(x)$ for all $x\in \partial \Delta^n$. Because this attaching map is a homotopy equivalence, this proves that there is a homotopy equivalence $(D_n,\partial D_n)\simeq (B^n,S^{n-1})$.

Finally, we want to choose a specific homotopy equivalence $\varphi_n : (B^n,S^{n-1}) \to (D_n,\partial D_n)$ whose restriction to $S^{n-1}$ is the $\psi_n$ constructed above. This is equivalent to a homotopy rel  $S^{n-2}$ between $\varphi_{n-1}$ and the map defined by $-\varphi_{n-1}(-\mathbf z)$. This exists because $(B^n,S^{n-1})$ is $(n-1)$-connected.
\end{proof}

\begin{thm}
\label{Main theorem}
If $f_1,\dots,f_n\in L^1[0,1]$, then there exists a function in $D_n$ that is orthogonal to all of $f_1,\dots, f_n$.
\end{thm}
\begin{proof}
Denote $\mathbf f\eqdef(f_1,\dots,f_n)\in L^1([0,1],\R^n)$. Define 
\[
\Phi :B^n\to \R^n
\]
by $\Phi(\mathbf z) \eqdef \langle \mathbf f,\varphi_n(\mathbf z)\rangle$. We need to prove that $0$ is in the image of $\Phi$.

We will prove this by contradiction, so suppose that $\Phi(B^n)\subset \R^n_*$, where $\R^n_*$ is the set of nonzero vectors. Because $B^n$ is contractible, $\Phi : S^{n-1}\to \R^n_*$ is homotopic to a constant map. This means that $\pi\circ \Phi : S^{n-1}\to S^{n-1}$ (where $\pi:\R^n_*\to S^{n-1}$, $\pi(\mathbf z) = \mathbf z/\Norm{\mathbf z}$) has degree $0$.

On the other hand, for $\mathbf z\in S^{n-1}$, $\Phi(-\mathbf z)=-\Phi(\mathbf z)$. This is the hypothesis of Theorem 3 in \cite{Whitt}, which states that the degree of $\pi\circ \Phi : S^{n-1}\to S^{n-1}$ must then be odd. 

This is a contradiction, therefore for some $\mathbf z\in B^n$,
\[
0 = \Phi(\mathbf z) = \langle \mathbf f,\varphi_n(\mathbf z)\rangle .
\]
This $\varphi_n(\mathbf z)\in D_n \subset L^\infty[0,1]$ is the desired function.
\end{proof}

\section{The Polynomial Case}
\label{Sect:Poly}
Now, we return to the more general Question~\ref{Generalized Question} but in the specific example where the $n$ functions are $1, x, x^2, \dots, x^{n-1}$ (or any other linearly independent set of linear combinations of these functions).   For this example, we will show existence and uniqueness and also compute the switch function.

For any $\lambda\in(-1,1)$, we wish to find a switch function, $\sigma$, with no more than $n$ switch points, such that
\[
0 = \langle p,\sigma-\lambda\rangle
\]
for any polynomial,  $p$, of degree $\leq n-1$. Note that if $\sigma=\Sigma_n(\mathbf x)$, then this condition is equivalent to the system of equations,
\begin{equation}
\label{thetasumlow}
x_n^k-x_{n-1}^k+x_{n-2}^k- \dots + (-1)^{n-1}x_1^k=\theta, \quad (1 \le k \le n) ,
\end{equation}
in which $\theta=(1+(-1)^{n-1}\lambda)/2 \in (0,1)$. 

We first give a simple, explicit solution to the computation part of Question~\ref{Main Question} (the case $\theta=\frac12$):
\begin{prop}
\[
x_j = \cos^2\left(\frac{[n+1-j]\pi}{2n+2}\right)
\]
is a solution of eq.~\eqref{thetasumlow} for $\theta=\frac12$ ($\lambda=0$) hence
\[
0 = \langle p,\Sigma_n(\mathbf x)\rangle
\]
for any polynomial, $p$, of degree $\leq n-1$.
\end{prop}
\begin{proof}
First notice that for $k\in\Z$, 
\[
\sum_{j=1}^n (-1)^{j-1}\cos\left(\frac{jk\pi}{n+1}\right) = \operatorname{Re}\frac{e^{ik\pi/(n+1)}-(-1)^{k+n}}{1+e^{ik\pi/(n+1)}} 
= \begin{cases}
0 & k+n \text{ even}\\
1 & k+n\text{ odd},
\end{cases}
\]
because, for $k+n$ is even, we have
\[
\operatorname{Re} i \tan\left(\frac{k\pi}{2n+2}\right) = 0 .
\]
Next, by the binomial theorem,
\[
\cos^{2k}t = 2^{-2k}\sum_{m=-k}^{k}\binom{2k}{k-m}e^{2imt}
=2^{-2k}\sum_{m=-k}^{k}\binom{2k}{k-m}\cos 2mt ,
\]
and by putting $t=0,\frac\pi2$ in this identity, we see that 
\[
1 = 2^{-2k}\sum_{m=-k}^{k}\binom{2k}{k-m}
\]
and
\[
0 = 2^{-2k}\sum_{m=-k}^{k}\binom{2k}{k-m} (-1)^m ,
\]
hence
\[
 2^{-2k}\sum_{k\;\mathrm{odd}}\binom{2k}{k-m} = 2^{-2k}\sum_{k\;\mathrm{even}} \binom{2k}{k-m} = \frac12 .
\]
We need to compute
\begin{multline*}
x_n^k-x_{n-1}^k+x_{n-2}^k- \dots + (-1)^{n-1}x_1^k 
= \sum_{j=1}^n (-1)^{n-j} x_j^k \\
= \sum_{j=1}^n (-1)^{n-j} \cos^{2k}\left(\frac{[n+1-j]\pi}{2n+2}\right) 
= \sum_{j=1}^n (-1)^{j-1} \cos^{2k}\left(\frac{j\pi}{2n+2}\right)\\
= \sum_{j=1}^n (-1)^{j-1}2^{-2k}\sum_{m=-k}^{k}\binom{2k}{k-m}\cos \left(\frac{mj\pi}{n+1}\right) 
\end{multline*}
By the first identity, this equals
\[
2^{-2k}\sum_{k+n\;\mathrm{odd}} \binom{2k}{k-m} = \frac12 ,
\]
which verifies eq.~\eqref{thetasumlow} in this case.
\end{proof}

For Question~\ref{Generalized Question} (general $\theta$) our strategy is to compute a polynomial whose roots are $x_1,\dots,x_n$.

\subsection{Preliminary remarks} 
It is helpful to consider first how one can solve the system of equations 
\begin{equation}
\label{thetasumplus}
\xi_n^k+\xi_{n-1}^k+\xi_{n-2}^k+ \dots + \xi_1^k=\theta, \quad (1 \le k \le n)
\end{equation}
which resemble equations \eqref{thetasumlow} except that, in place of the alternating signs, all the signs are positive.   To do this, we may seek an order-$n$ polynomial, $P(x)$ such that $\xi_1 \dots \xi_n$ solve \eqref{thetasumplus} if and only if they are its roots. We can assume that this is monic, so that
\[
P(x) = \prod_{j=1}^n (x-\xi_j) .
\]

\begin{definition}
For any polynomial, $Q$, of degree $m$ the \emph{reciprocal polynomial} is
\begin{equation*}
Q^*(x) \eqdef x^mQ(1/x)
\end{equation*}
so that the coefficients of $Q^*$ are those of $Q$ in reverse order. 
\end{definition} 

We use this notation throughout this section.   
 
With this notation, and using the Taylor expansion of the natural logarithm, 
\begin{align*}
P^*(x)&=\prod_{j=1}^n(1-\xi_j x) = \exp\left\{\sum_{j=1}^n\log(1-\xi_j x)\right\} \\
&= \exp(-s_1 x-\half s_2x^2 - \tfrac{1}{3}s_3x^3 - \dots)
\end{align*}
where
\begin{equation*}
s_k=\xi_1^k+\xi_2^k+\dots +\xi_n^k.
\end{equation*}
Equations \eqref{thetasumplus} state that $s_k=\theta$ for $k=1,\dots,n$, so
\begin{align}
P^*(x)&=\exp\left\{-\theta x - \tfrac12\theta x^2 - \dots - \tfrac1n\theta x^n + O(x^{n+1})\right\} \nonumber \\
\label{plus}
&=(1-x)^\theta + O(x^{n+1})\\
&=1-\theta x + \binom{\theta}{2}x^2-\binom{\theta}{3}x^3 + \dots (-1)^{n}\binom{\theta}{n}x^n + O(x^{n+1}) . \nonumber
\end{align}
Because $P^*$ is a polynomial of degree $n$,  it is clear that the $O(x^{n+1})$ term in the last line must actually vanish.   So $P(x)$ is determined uniquely to be 
\begin{equation*}
P(x)=x^n-\theta x^{n-1} + \binom{\theta}{2}x^{n-2}-\binom{\theta}{3}x^{n-3} + \dots (-1)^{n}\binom{\theta}{n}
\end{equation*}
and we may conclude that \eqref{thetasumplus} has a  solution $(\xi_1, \dots \xi_n)$ -- unique up to permutations -- consisting of the roots of this $P(x)$.

\subsection{Construction}
  To adapt the method explained above to Equations~\eqref{thetasumlow} we must deal with the alternating signs.

\begin{lem}
\label{Signs}  
Suppose that $0<\theta<1$ and $\delta_i=\pm1$ for $i=1,\dots,n$. If  $0 \le x_0 \le x_1 \le \dots \le x_n \le 1$ satisfy
\begin{equation}
\label{altdelta}
\delta_n x_n^k + \delta_{n-1} x_{n-1}^k + \dots + \delta_1x_1^k=\theta \qquad (1 \le k \le n)
\end{equation}
then $\delta_i=(-1)^{n-i}$ and $0<x_1<x_2<\dots<x_n<1$.
\end{lem}
\begin{proof}
Define $x_0=0$, $x_{n+1}=1$, and $\delta_{n+1}=-1$.  

Note that eqs.~\eqref{altdelta} mean precisely that for any polynomial, $D$, of degree $\leq n-1$,
\[
\sum_{i=1}^n \delta_i x_i D(x_i) = \theta D(1) .
\]

For some polynomial $D$ (of any degree) and a set $T\subset\{1,\dots,n\}$, consider the sum
\[
\sum_{i\in T} \delta_i x_i D(x_i) .
\]
Starting from $T=\{1,\dots,n\}$, we can simplify this set in two ways without changing the sum:
\begin{itemize}
\item
If $i\in T$ and $x_i=0$, then remove $i$ from $T$ (because the $i$ term contributes nothing to the sum).
\item
If $i,j\in T$ with $x_i=x_j$ and $\delta_i\neq\delta_j$ then remove $i$ and $j$ from $T$ (because the $i$ and $j$ terms cancel).
\end{itemize}
Apply these steps iteratively until we have a set $T$ such that for all $i,j\in T$, $x_i\neq0$ and $x_i=x_j\implies\delta_i=\delta_j$, and
\[
\sum_{i\in T} \delta_i x_i D(x_i) = \sum_{i=1}^{n} \delta_i x_i D(x_i) .
\]
Now consider 2 possible cases.

\emph{Case 1: $x_j=\delta_j=1$ for some $j\in T$.} Assume without loss of generality that that is $n\in T$.
For each consecutive $i,j\in T$ such that $x_i<x_j$ and $\delta_i\neq\delta_j$, choose a number in the open interval $(x_i,x_j)$. Construct a polynomial $D$ with these numbers as its (simple) roots and choose the overall sign so that $D(1)>0$. In this way, we have $\delta_i D(x_i)>0$ for all $i\in T$, so
\begin{align*}
0 &< (1-\theta)D(1) + \sum_{i\in T,\neq n} \delta_i x_i D(x_i) =-\theta D(1) + \sum_{i\in T} \delta_i x_i D(x_i) \\
&\quad= -\theta D(1) + \sum_{i=1}^n \delta_i x_i D(x_i) .
\end{align*}

\emph{Case 2: Otherwise.}
For each consecutive $i,j\in T\cup\{n+1\}$ such that $x_i<x_j$ and $\delta_i\neq\delta_j$, choose a number in the open interval $(x_i,x_j)$. Construct a polynomial $D$ with these numbers as its (simple) roots and choose the overall sign so that $D(1)<0$. In this way, we have $\delta_i D(x_i)>0$ for all $i\in T\cup\{n+1\}$, so
\[
0 < -\theta D(1) + \sum_{i\in T} \delta_i x_i D(x_i) = -\theta D(1) + \sum_{i=1}^n \delta_i x_i D(x_i) .
\]

In either case, if $\deg D\leq n-1$, then the last expression is $0$, which is a contradiction. Therefore $\deg D=n$, but this requires $T=\{1,\dots,n\}$, with $x_i<x_{i+1}$ and $\delta_i\neq \delta_{i+1}$ for all $i=1,\dots,n$.
\end{proof}

This is equivalent to:
\begin{cor}
\label{Plusminus}
If $0<\theta<1$, and 
\begin{equation}
\label{thetasumm}
\xi_1^k+\xi_2^k+ \dots + \xi_{n-m}^k - (\eta_1^k + \eta_2^k + \dots + \eta_m^k)=\theta, \qquad (1 \le k \le n)
\end{equation}
such that $\xi_i,\eta_i\in[0,1]$ for all $i$, then $m=\lfloor\frac{n}2\rfloor$ and the sets of $\xi_i$'s and $\eta_i$'s interleave in the sense that if they are labelled in increasing order then 
\begin{equation}
\label{interleave}
\begin{split}
0<\eta_1 < \xi_1 < \eta_2 < \xi_2 <  \dots < \eta_m < \xi_m<1 \qquad& (n=2m) \\
0<\xi_1 < \eta_1 < \xi_2 < \eta_2 <  \dots < \eta_m < \xi_{m+1}<1 \qquad& (n=2m+1).
\end{split}
\end{equation}
 \end{cor}

We will also need a generalization of these results  later for the proof of Theorem~\ref{Jacobi}:
\begin{cor}
\label{Sequence}
The same result applies if we replace the condition $1 \le k \le n$ in \eqref{altdelta} and \eqref{thetasumm} by the condition that $k-1\in\mathcal I$ for some $\mathcal I\subset\N$ of size $\Abs{\mathcal I}=n$.
\end{cor}
\begin{proof}
All we need to generalize the proof is to show that for any $m\leq n-1$ positive real numbers, there exists a polynomial $D$, of degree $m$, whose exponents are all in $\mathcal I$, and whose positive roots are simple and are precisely the given numbers.

To construct such a polynomial, consider an arbitrary monic polynomial that uses the first $m+1$ numbers in $\mathcal I$ as its exponents. The condition on the roots gives a system of linear equations for the $m$ undetermined coefficients. This has a unique solution, because of the linear independence of different powers of $x$. Descartes' rule of signs shows that the number of positive roots of $D$ is at most the number of sign changes in the coefficients of $D$, which is at most $m$; therefore there can be no other positive roots.
\end{proof}

In view of Corollary~\ref{Plusminus},  the problem of finding solutions to eqs.~\eqref{thetasumlow} that satisfy the condition $0 \le x_1 \le x_2 \le \dots \le x_n \le 1$ is immediately solved once one finds a solution to the equations \eqref{thetasumm} for which all the $\xi_i$, $i=1 \dots n-m$ and the $\eta_j, j=1\dots m$ lie in $[0,1]$.  

\begin{lem}
\label{Roots}
Let $P^+$ and $P^-$ be monic polynomials of degrees $n-m$ and $m$.
The roots of $P^+$ and $P^-$ will be a solution to eqs.~\eqref{thetasumm} if and only if
\begin{equation}
\label{PplusoverPminus}
\frac{{P^+}^*(x)}{{P^-}^*(x)}=(1-x)^\theta + O(x^{n+1}).
\end{equation}
\end{lem}
(We remark that ${P^+}^*(x)$ and ${P^-}^*(x)$ are what are called Pad\'e approximants \cite{Pade} to $(1-x)^\theta$.)
\begin{proof}
Write the factorizations of these polynomials as
\[
P^+(x)=\prod_{i=1}^{n-m} (x-\xi_i) \qquad\text{and}\qquad P^-(x) = \prod_{j=1}^m (x-\eta_j).
\]
Generalizing our preliminary example, this means that
\begin{align*}
\frac{{P^+}^*(x)}{{P^-}^*(x)} &= \frac{\prod_{i=1}^{n-m} (1-\xi_ix)}{\prod_{j=1}^m (1-\eta_jx)}
= \exp\left\{\sum_{i=1}^{n-m} \log (1-\xi_ix) - \sum_{j=1}^m \log(1-\eta_jx)\right\} \\
&= \exp\left\{-\sum_{k=1}^n\frac1k s_k x^k\right\} + O(x^{n+1}) ,
\end{align*}
where
\[
s_k \eqdef \xi_1^k+\xi_2^k+ \dots + \xi_{n-m}^k - (\eta_1^k + \eta_2^k + \dots + \eta_m^k) .
\]
On the other hand,
\[
 (1-x)^\theta =  \exp\left\{-\sum_{k=1}^n\frac1k \theta x^k\right\} + O(x^{n+1}) ,
\]
so $s_k=\theta$ for $k=1,\dots,n$ (which is eqs.~\eqref{thetasumm}) if and only if eq.~\eqref{PplusoverPminus} holds.
\end{proof}

\begin{definition}
\begin{equation}
\label{P}
P_m(x,\theta)= \sum_{i=0}^m (-1)^i \frac{m! (2m-i)!}{(2m)! (m-i)!} \binom{m+\theta}{i} x^{m-i}
\end{equation}
\begin{equation}
\label{Q}
Q_{m+1}(x,\theta)= \sum_{i=0}^{m+1} (-1)^i \frac{(m+1)!(2m+1-i)!}{(2m+1)!(m+1-i)!}\binom{m+\theta}{i} x^{m+1-i}
\end{equation}
\begin{equation}
\label{R}
R_m(x,\theta)= \sum_{i=0}^{m} (-1)^i \frac{m!(2m+1-i)!}{(2m+1)!(m-i)!}\binom{m+1+\theta}{i} x^{m-i} .
\end{equation}
\end{definition}

\begin{thm}
\label{Explicit}
For $\lambda\in(-1,1)$ and $n\in\N$, 
\[
\sigma=\Sigma_n(\mathbf x)
\]
satisfies  
\[
0 = \langle p,\sigma-\lambda\rangle
\]
for any polynomial of degree $\leq n-1$ if and only if the $x_i$ are the roots (in order) of the polynomials of $x$:
\begin{equation*}
P_m(x,\theta)\text{ and }P_m(x,-\theta), \text{ for }n=2m\text{ and }\theta=\tfrac{1-\lambda}2,
\end{equation*}
\begin{equation*}
Q_{m+1}(x,\theta)\text{ and }R_m(x,-\theta),\text{ for } n=2m+1\text{ and }\theta=\tfrac{1+\lambda}2.
\end{equation*}
In each case $x_1, x_3,\dots$ are the roots of the first polynomial.
\end{thm}
\begin{proof}
Consider the case $n=2m$. By Lemma~\ref{Roots} and Corollary~\ref{Plusminus}, it is sufficient to show that $P^\pm(x)=P_m(x,\pm\theta)$ satisfy eq.~\eqref{PplusoverPminus} and have all of their roots between $0$ and $1$.

From eq.~\eqref{P} we see that
\begin{equation}
\label{hyperpoly}
P_m^*(x,\theta) = 1 - \frac{m}{2m}\cdot\frac{m+\theta}{1!}+\frac{m(m-1)}{2m(2m-1)}\binom{m+\theta}{2}x^2-\dots + (-1)^m\binom{2m}{m}^{-1}\binom{m+\theta}{m}x^m,
\end{equation}
which is a hypergeometric function, $P_m^*(x,\theta) = {}_2F_1(-m,-m-\theta;-2m;x)$. In this case, the power series terminates at the $m$'th term, so that albeit $c=-2m$ is a negative integer, the denominators do not vanish.  

The roots of the polynomial \eqref{hyperpoly} all lie in $(1,\infty)$; the roots of the reciprocal polynomial, $P_m(x,\theta)$, are the reciprocals of these,  and hence, as required,  lie in $(0,1)$.  We refer to \cite[Theorem 2.3 (iii)]{Zeros1} in which the parameter $k$ is seen to be $0$.   To cope with $P_m^*(x,-\theta)$, we employ \cite[Theorem 2.3 (ii)]{Zeros1} (see also \cite{Zeros2}) when the parameter again denoted by $k$ equals $m$ to see that all roots lie in $(0,1)$.

In general, the hypergeometric function
\begin{equation}
\label{hyper}
{}_2F_1(a,b;c;x)=1+\frac{ab}{1!c}x+\frac{a(a+1)b(b+1)}{2!c(c+1)}x^2 + \dots 
\end{equation}
satisfies the Gauss differential equation
\begin{equation*}
x(1-x)F''+\{c-(1+a+b)x\}F'-abF=0.
\end{equation*}
The function $W(x)=(1-x)^{-\theta}P_m^*(x,\theta)$ satisfies the differential equation
\begin{equation}
\label{Wode}
x(1-x)W''+\{-2m-(1-2m+\theta)x\}W'-m(m-\theta)W=0
\end{equation}
which is a Gauss equation with parameters $-m+\theta, -m, -2m$.    The indicial equation for \eqref{Wode}  has roots $0,1+2m$.  Since $P_m^*(x,\theta)$ is a polynomial and $0<\theta < 1$, $W$ does not involve logarithms and has an infinite, convergent power series expansion comprising two parts: $P_m^*(x,-\theta)$ and then a series whose first term involves $x^{2m+1}$.  We refer to 
\cite[p.~286]{WW} to find that, with a suitable constant $c_m(\theta)$, we have
\begin{equation*}
(1-x)^{-\theta}P_m^*(x,\theta)=P_m^*(x,-\theta) + c_m(\theta)x^{2m+1}{}_2F_1(m+1,m+1+\theta,2m+2;x),
\end{equation*}
which satisfies \eqref{PplusoverPminus} as we wished to show.  

We do not need the value of $c_m(\theta)$ but remark that this may be deduced from \cite[p.\ 299, Ex.~18.]{WW}.  For suitable values of the parameters, this gives an asymptotic formula for the hypergeometric function as $x\rightarrow 1$.

The odd $n$ case is similar and is left to the reader.
\end{proof}
Another way of expressing this is that 
\[
\sigma(x) = \sgn\left\{P_m(x,\theta)P(x,-\theta)\right\}
\]
and so on. Although this formula gives the value $0$ at some points, it still defines a switch function within $L^\infty[0,1]$.

For example, let $n=5, \theta=1/3$.  From  \eqref{Q} and \eqref{R},
\begin{equation*}
Q_3(x,1/3)=x^3-\frac{7}{5}x^2+\frac{7}{15}x-\frac{7}{405}, \quad R_2(x,-1/3)=x^2-\frac{16}{15}x+\frac{2}{9}.
\end{equation*}
Therefore, $x_1 \doteq .0422244245$, $x_2 \doteq .2838895075$, $x_3 \doteq .4518343712$, $x_4 \doteq .7827771591$,
$x_5 \doteq .9050412043$.

\subsection{Uniqueness}
We now answer the uniqueness part of Questions~\ref{Main Question} and \ref{Generalized Question} in this polynomial case.
 
\begin{thm}
\label{Uniqueness}
For $\lambda\in(-1,1)$ and $n\in\N$, the switch function $\sigma$ described in Theorem~\ref{Explicit} is the \emph{unique} $\sigma\in D_n$ such that $0 = \langle p,\sigma-\lambda\rangle$ for any polynomial $p$ of degree $\leq n-1$.
\end{thm}
\begin{proof}
 We need to show that for $\theta\in(0,1)$, the solution of eqs.~\eqref{thetasumlow} with $\mathbf x\in\Delta^n$ is unique. Any solution determines polynomials $P^\pm$ satisfying eq.~\eqref{PplusoverPminus}, and the polynomials determine an ordered solution uniquely. It is therefore sufficient to show uniqueness for the solution of eq.~\eqref{PplusoverPminus}.

To prove the uniquess of $P^\pm$, write
 \begin{gather*}
{P^+}^*(x) = 1 - a_1x + a_2 x^2 - \dots + (-1)^{n-m} a_{n-m}x^{n-m}, \\
 {P^-}^*(x) = 1 -  b_1x + b_2 x^2 - \dots  + (-1)^m b_mx^m .
\end{gather*}
Cross multiplying in \eqref{PplusoverPminus} and equating coefficients of powers of $-x$ up to $(-x)^n$, we easily find that \eqref{PplusoverPminus} is equivalent to the block matrix equation 
\begin{equation}
\label{block}
\begin{pmatrix} I & -T \\ 0 & -K \end{pmatrix}\begin{pmatrix} A \\ B \end{pmatrix} = C,
\end{equation}
where $I$ is the $(n-m)\times (n-m)$ identity matrix, $0$ the $m\times (n-m)$ zero matrix, $A$, $B$, and $C$ are the $(n-m)\times 1$, $m\times 1$, and $n\times 1$ matrices
\begin{equation*}
A=\begin{pmatrix} a_1\\a_2\\ \vdots \\a_{n-m}\end{pmatrix}, \quad B=\begin{pmatrix} b_1\\b_2\\ \vdots \\ b_m\end{pmatrix}; \quad C=\begin{pmatrix} c_1\\ c_2 \\ \vdots   \\ c_n\end{pmatrix},
\end{equation*}
where 
\begin{equation*}
c_i=\binom{\theta}{i}
\end{equation*}
is the coefficient of $(-x)^i$ in the binomial expansion, \eqref{plus} of $(1-x)^\theta$;
$T$ is the $(n-m)\times m$ matrix, given, when $n=2m$, by
\begin{equation*}
T = \begin{pmatrix} 1 & 0 & 0 & \dots & 0 & 0 \\ 
c_1 & 1 & 0 & \dots &  0 & 0 \\c_2 & c_1 & 1 & \dots &  0 & 0\\ 
\vdots & \vdots & \vdots& \ddots  & \vdots & \vdots \\ 
c_{m-2} & c_{m-3} & c_{m-4} &\dots  & 1 & 0\\
c_{m-1} & c_{m-2} & c_{m-3} & \dots & c_1 & 1
\end{pmatrix}, 
\end{equation*}
and, when $n=2m+1$, by
\begin{equation*}
T =  \begin{pmatrix} 
 1 & 0 & 0 & \dots & 0 & 0 \\ 
c_1 & 1 & 0 & \dots &  0 & 0 \\c_2 & c_1 & 1 & \dots &  0 & 0\\ 
\vdots & \vdots & \vdots& \ddots  & \vdots & \vdots \\ 
c_{m-2} & c_{m-3} & c_{m-4} &\dots  & 1 & 0\\
c_{m-1} & c_{m-2} & c_{m-3} & \dots & c_1 & 1\\
c_m & c_{m-1} & c_{m-2} &\dots & c_2 & c_1 \end{pmatrix},
\end{equation*}
while (both when $n$ is even and when $n$ is odd) $K$ is the $m \times m$ matrix,
\begin{equation}
\label{K}
K = \begin{pmatrix} c_{n-m} & c_{n-m-1} & \dots & c_{n-2m+1} \\c_{n-m+1} & c_{n-m} & \dots & 
c_{n -2m+2} \\ \vdots & \vdots & \ddots & \vdots \\ c_{n-1} & c_{n-2} & \dots  & c_{n-m}\end{pmatrix}.
\end{equation}

Clearly, the determinant of the $2\times 2$ block matrix in \eqref{block} is $(-1)^m \det K$.   We show in the appendix (Propositions \ref{DetKeven} and \ref{DetKodd}) that this is never zero for $\theta\in(0,1)$.     So the matrix equation \eqref{block} has a unique solution and therefore the polynomials, $P^+$ and $P^-$ are determined uniquely by eq.~\eqref{PplusoverPminus}.  
\end{proof}

The reader may wonder why we proved existence by verifying that the polynomials given in the statement of Theorem~\ref{Explicit} satisfy $\eqref{thetasumlow}$ rather than by deriving these polynomials from those equations, and they may further wonder how we came to know that these were the right polynomials to try.   To address these questions, we remark that in principle it must, of course, be possible to solve \eqref{block} and thereby derive the explicit formulae \eqref{P} for $P_m(x, \theta)$ and $P_m(x,-\theta)$  in the case $n$ is even and \eqref{Q} and \eqref{R} for $Q_{m+1}(x, \theta)$ and $R_m(x,\theta)$  in the case $n$ is odd, for $P^+$ and $P^-$, as given in the statement of Theorem~\ref{Explicit}.  However, in practice, such a direct derivation of eqs.~\eqref{P}, \eqref{Q}, and \eqref{R} seems difficult.   It is possible, though, to solve them for small $n$ and thereby to be able to guess the form of these polynomials for all $n$.  Alternatively, one can guess them after directly solving the equation \eqref{thetasumlow} for small $n$ and this is what we did.

\section{The Even Polynomial Case}
\label{Sect:Even}
Now we consider Question~\ref{Generalized Question} for the functions $1,x^2,x^4,\dots, x^{2n-2}$. We will show existence by computing the switch function, and we shall prove its uniqueness.

Recall that the degree $n$ Jacobi polynomial \cite{Freud} for parameters $\theta$ and $-\theta$ is
\begin{equation}
\label{Jac}
\begin{split}
J_n(\theta, -\theta; x) &= (-1)^n\left(\frac{1-x}{1+x}\right)^\theta\frac{n!}{(2n)!}\frac{d^n}{dx^n}\{ (1+x)^{n+\theta}(1-x)^{n-\theta}\}\\
&=x^n-\theta x^{n-1} -\frac{(n-1)(n-2\theta^2)}{2(2n-1)}x^{n-2}+\dots .
\end{split}
\end{equation}
This polynomial has $n$ distinct (non-zero) roots, $\zeta_i$ on $(-1,1)$, being one of a sequence of orthogonal polynomials on this interval.  

\begin{thm}
\label{Jacobi}
Let $n\in\N$, $-1\leq\lambda\leq1$, and again $\theta=(1+(-1)^{n-1}\lambda)/2$.
A switch function $\sigma\in D_n$ satisfying 
\[
0=\langle p,\sigma-\lambda\rangle
\] 
for any even polynomial, $p$, of degree $\leq2n-2$ is given by $\sigma=\Sigma_n(\mathbf x)$, where $x_i=\Abs{\zeta_i}$, and $\zeta_i$ are the roots of the Jacobi polynomial $J_n(\theta,-\theta;x)$  ordered by absolute value.
\end{thm}
\begin{proof}
Note that if $\sigma=\Sigma_n(x)$, then this condition on $\sigma$ is equivalent to the system of equations,
\begin{equation}
\label{thetasumhigh}
x_1^{2k-1}-x_2^{2k-1} + x_3^{2k-1} - \dots + (-1)^{n-1}x_n^{2k-1}=\theta, \qquad (1 \le k \le n) .
\end{equation}

Consider the system of equations
\begin{equation}
\label{positive}
\theta = \sum_{i=1}^n \zeta_i^{2k-1}
\end{equation}
for $1\leq k \leq n$. This is the same as eq.~\eqref{thetasumhigh}, but without negative signs. Because the exponents are all odd, this is equivalent to 
\[
\theta = \sum_{i=1}^n (\sgn \zeta_i)\Abs{\zeta_i}^{2k-1} .
\]
If we require $-1\leq \zeta_i\leq 1$ and label in order of increasing $\Abs{\zeta_i}$, then Corollary~\ref{Sequence} shows that $\sgn\zeta_i = (-1)^{n-i}$. Therefore any solution of eqs.~\eqref{positive} with $-1\leq\zeta_i\leq1$ gives a solution of eqs.~\eqref{thetasumhigh} by $x_i=\Abs{\zeta_i}$.

We can easily adapt the method explained in the preliminary remarks to see that  $\{\zeta_1, \zeta_2, \dots, \zeta_n\}$ will be a solution to eqs.~\eqref{positive} if and only if these are the roots of an order $n$ polynomial, $P$, such that
\begin{equation}
\label{even}
P^*(x) =  \exp(-s_1 x - \half s_2x^2 - \tfrac{1}{3}s_3 x^3 -\tfrac{1}{4} s_4 x^4 - \dots )
\text{ where }
s_{2k-1} = \theta \text{ for } 1 \le k \le n .
\end{equation}
Equation \eqref{even} is equivalent to requiring that $(1-x)^{-\theta}P^*(x)$  is an even function of $x$ up to and including order $x^{2n}$.

Now let $P(x)= J_n(\theta,\-\theta;x)$. We need to show that for $0<\theta<1$,
\[
G(x) \eqdef (1-x)^{-\theta}P^*(x)
\]
is even up to order $x^{2n}$.

We first notice that the Jacobi polynomial, $P(x)=J_n(\theta, -\theta;x)$ satisfies the differential equation
\begin{equation}
\label{Jdiff}
(1-x^2)P'' + (2\theta-2x)P' + n(n+1)P=0
\end{equation}
\eqref{Jdiff}, we deduce that the reciprocal polynomial, $P^*$, satisfies
\begin{equation*}
x(1-x^2){P^*}''+2[(n-1)x^2+\theta x-n]{P^*}'-n[(n-1)x+2\theta]P^*=0 ,
\end{equation*}
and $G$ satisfies
\begin{equation*}
x(1-x^2)G''+[-2n-(2\theta - 2n + 2)x^2]G'-(\theta-n)(\theta-n+1)x G=0.
\end{equation*}
This is a hypergeometric equation with independent variable $x^2$ and parameters $a=(\theta-n)/2$, $b=(\theta-n+1)/2$ and $c=-n+\half$.  Therefore \cite{WW} $G$ may be written in the form
\begin{equation}
\label{Geq}
G(x)={}_2F_1\left(\tfrac{\theta-n}{2}, \tfrac{\theta-n+1}{2}, -n+\half;x^2\right)+Cx^{2n+1}{}_2F_1\left(\tfrac{\theta+n+1}{2}, \tfrac{\theta+n+2}{2}, n+\tfrac{3}{2}; x^2\right)
\end{equation}
in which $C$ is a constant.    We refer to \cite[p.~299, Ex.~18]{WW} as in the proof of Theorem~\ref{Explicit}.   We see from \eqref{Geq} that $G(x)$ is indeed even to order $x^{2n}$.

Since the roots of  $J_n(\theta, -\theta; x)$ are in the interval $(-1,1)$, we may  conclude that when their absolute values are  labelled in increasing order, they will solve eqs.~\eqref{thetasumhigh} and be the switch points for the desired switch function. 
\end{proof}
For example, let $n=5$, $\theta=1/3$ as before.  The Jacobi polynomial in \eqref{Jac} is
\begin{equation*}
x^5- \frac{1}{3}x^4-\frac{86}{81}x^3+\frac{62}{243}x^2+\frac{157}{729}x-\frac{143}{6561}
\end{equation*}
and the roots are $\zeta_1\doteq.0948419$, $\zeta_2\doteq-.4571986$, $\zeta_3\doteq.6167796$, $\zeta_4\doteq-.8641519$, 
$\zeta_5\doteq.9430623$. The reader might care to try the resulting $x_i$ in \eqref{thetasumhigh}.

\subsection{Uniqueness}
\begin{thm}
\label{Uniqueness2}
 For all $n\in \mathbb{N}$ and $-1<\lambda<1$, the switch function $\sigma\in D_n$ such that $\sigma-\lambda$ is orthogonal to all even polynomials of degree $\leq 2n-2$ is unique.
\end{thm}
\begin{proof}
Uniqueness of $\sigma$ means uniqueness of the switch points. The roots of $P$ are $\zeta_i=(-1)^{n-i}x_i$, therefore the monic polynomial $P$ is uniquely determined by the switch points. It is therefore sufficient to show that $P$ is uniquely determined by eq.~\eqref{even}.

Write $P^*(x) = 1+a_1x+a_2x^2 + \dots a_{n}x^{n}$ and expand $(1-x)^{-\theta}P^*(x)$. We require the coefficients of odd powers of $x$ up to $x^{2n-1}$ to vanish. This is equivalent to the matrix equation:
\begin{equation}
\label{BAV}
BA=-V
\end{equation}
where, when $n$ is even,
$B$, $A$, and $V$ are the $n \times n$, $n\times 1$ and $n\times 1$ matrices
\begin{equation}
\label{Ba}
B=\begin{pmatrix} 
1 & 0 & 0 & 0  & \dots & 0 \\
 b_2 & b_1 & 1 & 0 &  \dots & 0 \\
 b_4 & b_3 & b_2 & b_1  & \dots & 0\\ 
 \vdots & \vdots & \vdots &\vdots & \vdots & \vdots\\ 
 b_{n-2} & b_{n-3} & b_{n-4} & b_{n-5} &  \dots & 0 \\ 
 b_n & b_{n-1} & b_{n-2} & b_{n-3} & \dots  & b_1\\  
 b_{n+2} & b_{n+1} & b_n & b_{n-1} & \dots  & b_3\\ 
 \vdots & \vdots & \vdots & \vdots & \vdots & \vdots  \\ 
 b_{2n-2} & b_{2n-3} & b_{2n-4} & b_{2n-5}  & \dots &  b_{n-1}
 \end{pmatrix}, \quad
A=\begin{pmatrix} a_1\\a_2\\ a_3  \\ a_4 \\ \vdots  \\a_n\end{pmatrix}, \quad V=\begin{pmatrix} b_1\\b_3\\ b_5 \\ \vdots \\ b_{n-1} \\ b_{n+1} \\ b_{n+3} \\ \vdots \\  b_{2n-1}\end{pmatrix},
\end{equation}
and when $n$ is odd, $B$, $A$, and $V$ take the form
\begin{equation}
\label{Bb}
B=\begin{pmatrix} 
1 & 0 & 0 & 0  & \dots & 0 \\
 b_2 & b_1 & 1 & 0 &  \dots & 0 \\
 b_4 & b_3 & b_2 & b_1  & \dots & 0\\ 
 \vdots & \vdots & \vdots &\vdots & \vdots & \vdots\\ 
 b_{n-3} & b_{n-4} & b_{n-5} & b_{n-6} &  \dots & 0 \\ 
 b_{n-1} & b_{n-2} & b_{n-3} & b_{n-4} & \dots  &1\\  
 b_{n+1} & b_{n} & b_{n-1} & b_{n-2} & \dots  & b_2\\ 
 \vdots & \vdots & \vdots & \vdots & \vdots & \vdots  \\ 
 b_{2n-2} & b_{2n-3} & b_{2n-4} & b_{2n-5}  & \dots &  b_{n-1}
 \end{pmatrix}, \quad
A=\begin{pmatrix} a_1\\a_2\\ a_3  \\ a_4 \\ \vdots \\ a_n\end{pmatrix}, \quad V=\begin{pmatrix} b_1\\b_3\\ b_5 \\ \vdots \\ b_{n-2} \\ b_{n} \\ b_{n+2} \\ \vdots  \\ b_{2n-1}\end{pmatrix},
\end{equation}
where 
\begin{equation*}
b_m= \frac{\theta(\theta +1) \dots (\theta +m-1)}{m!}.
\end{equation*}
We show in the appendix (Prop.~\ref{DetB}) that $\det B$ never vanishes for $\theta\in (0,1)$ whereupon the matrix equation \eqref{BAV} will have a unique solution and therefore  there will be a unique polynomial $P$ that determines $\sigma$.
\end{proof}

\section{Sines}
Consider the functions $f_k(t) = \sin \frac{k\pi}{2}t$ for $k=1,\dots,n$. The problem of finding switch functions in this case reduces to the polynomial problem by a simple change of variables.

Recall that the Chebyshev polynomial $T_k$ is a degree $k$ polynomial satisfying 
\[
T_k(\cos x) = \cos kx .
\]
Suppose that $\sigma$ is a switch function satisfying $0=\langle f_k,\sigma-\lambda\rangle$. The change of variables $s = \cos\frac\pi2t$ gives $T_k(s) = \cos \frac{k\pi}2t$, and so
\begin{align*}
0 &= \int_0^1 \left(\sigma(t)-\lambda\right) \sin \tfrac{k\pi}2t\,dt \\
&= -\frac2{k\pi} \int_0^1 \left(\sigma(t)-\lambda\right) \frac{d}{dt}\cos \tfrac{k\pi}2t\,dt\\
&= \frac2{k\pi} \int_0^1\left(\sigma(\tfrac2\pi\cos^{-1} s) - \lambda\right) T'_k(s)\,ds .
\end{align*}
Now, $T_k'$ is of degree $k-1$, and these form a basis of polynomials. Therefore, the switch function
\[
s\mapsto \sigma(\tfrac2\pi\cos^{-1} s)
\]
satisfies our problem for polynomials of degree up to $n-1$.

Note, however, that the order of positive and negative values has been reversed.

So, suppose that $(x,1,\dots,x_n)$ is the solution for the polynomial problem with parameter $(-1)^n\lambda$, and let $y_j=\frac2\pi\cos^{-1}x_{n+1-j}$. The corresponding switch functions are related by 
\[
\Sigma_n(\mathbf y)(t) = (-1)^n \Sigma_n(\mathbf x)(\cos\tfrac\pi2t) ,
\]
so $\mathbf y$ is a solution the the sine problem with parameter $\lambda$.

\subsection*{Acknowledgments}
BSK thanks Michael M.\ Kay for very helpful conversations.

\appendix
\section{Appendix}
In this appendix, we show that the determinants $\det K$ ($K$ as in eq. \eqref{K}) and $\det B$ ($B$ as in eqs.
\eqref{Ba}, \eqref{Bb}) never vanish for $\theta \in (0,1)$.   In fact, in order to show this, we evaluate them in full.

We evaluate $\det K$ separately for $n$ even and for $n$ odd.

\begin{prop}
\label{DetKeven}
For $n=2m$, the determinant of the matrix in eq.~\eqref{K} is
\beq
\label{detKeven}
\det K = \frac{\theta^m(\theta^2-1)^{m-1}(\theta^2-4)^{m-2}\dots(\theta^2-[m-1]^2)}{m^m(m^2-1)^{m-1}\dots(m^2-[m-1]^2)}  .
\eeq
\end{prop}
\begin{proof}
For $n=2m$, 
\[
\det K = 
\begin{vmatrix}
\binom\theta{m} & \binom\theta{m-1} & \dots & \binom\theta{1} \\
\binom\theta{m+1} & \binom\theta{m} & \dots & \binom\theta{2} \\
\vdots & \vdots & \ddots & \vdots \\
\binom\theta{2m-1} & \binom\theta{2m-2} & \dots & \binom\theta{m}
\end{vmatrix}
\]
is a polynomial in $\theta$ of degree $m^2$. As an element of $\mathbb{Q}[\theta]$, the first row is divisible by  $\binom\theta{1}$, the second by $\binom\theta{2}$, $\dots$ and the $m$'th by $\binom\theta{m}$, hence $\det K$ has a factor
\[
\theta^m (\theta-1)^{m-1}\dots(\theta-m+1)
\]
of degree $\frac12m(m+1)$, so we want to find $\frac12m(m-1) = 1+2+\dots+(m-1)$ further factors.

The determinant is invariant under row operations. Add the $m-1$'st row to the $m$'th row, the $m-2$'nd row to the $m-1$'st row, $\dots$ and the 1st row to the second row. This shows that
\begin{equation}
\label{first row}
\det K = 
\begin{vmatrix}
\binom\theta{m} & \binom\theta{m-1} & \dots & \binom\theta{1} \\
\binom{\theta+1}{m+1} & \binom{\theta+1}{m} & \dots & \binom{\theta+1}{2} \\
\vdots & \vdots & \ddots & \vdots \\
\binom{\theta+1}{2m-1} & \binom{\theta+1}{2m-2} & \dots & \binom{\theta+1}{m}
\end{vmatrix} .
\end{equation}
The last $m-1$ rows are divisible by  $\theta+1$.

Starting from \eqref{first row}, add the $m-1$'st row to the $m$'th row, \dots\ and the 2nd row to the 3rd row. This shows that
\[
\det K = 
\begin{vmatrix}
\binom\theta{m} & \binom\theta{m-1} & \binom\theta{m-2} &\dots & \binom\theta{1} \\
\binom{\theta+1}{m+1} & \binom{\theta+1}{m} & \binom{\theta+1}{m-1} & \dots & \binom{\theta+1}{2} \\
\binom{\theta+2}{m+2} & \binom{\theta+2}{m+1} & \binom{\theta+2}{m} & \dots & \binom{\theta+2}{3} \\
\vdots & \vdots & \vdots &  \ddots & \vdots \\
\binom{\theta+2}{2m-1} & \binom{\theta+2}{2m-2} & \binom{\theta+2}{2m-3} &\dots & \binom{\theta+2}{m}
\end{vmatrix} .
\]
The last $m-2$ rows are divisible by  $\theta+2$.

Carrying on like this, we find that $\det K$ is proportional to
\[
\theta^m(\theta^2-1)^{m-1}(\theta^2-4)^{m-2}\dots(\theta^2-[m-1]^2) .
\] 
To determine the constant of proportionality, note that if $\theta=m$, then  $K$ is an upper triangular matrix with $1$'s on the diagonal, and hence $\det K=1$. This gives eq.~\eqref{detKeven}.

Alternatively, this is  a special case of Theorem 1 of \cite{RHPK} (paraphrased below as Thm.~\ref{RHPKthm}).  It corresponds to $N=m$ and the set of integers $L=\{m+1,m+2,m+3,...,2m\}$.
\end{proof}

\begin{prop}
\label{DetKodd}
For $n=2m+1$, the determinant of the matrix in eq.~\eqref{K} is
\beq
\label{detKodd}
\det K = \frac{\theta^m(\theta-1)^m(\theta+1)^{m-1}(\theta-2)^{m-1}\dots(\theta+m-1)(\theta-m)}{(m+1)^m m^m(m+2)^{m-1}(m-1)^{m-1}\dots1\cdot 2m} .
\eeq
\end{prop}
\begin{proof}
For case $n=2m+1$, 
\[
\det K = 
\begin{vmatrix}
\binom\theta{m+1} & \binom\theta{m} & \dots & \binom\theta{2} \\
\binom\theta{m+2} & \binom\theta{m+1} & \dots & \binom\theta{3} \\
\vdots & \vdots & \ddots & \vdots \\
\binom\theta{2m} & \binom\theta{2m-1} & \dots & \binom\theta{m+1}
\end{vmatrix}
\]
has degree $m^2+m$. The rows are divisible by $\binom\theta2,\binom\theta3,\dots,\binom\theta{m+1}$, so $\det K$ has a factor
\[
\theta^m\cdot (\theta-1)^m(\theta-2)^{m-1}\dots(\theta-m) 
\]
of degree $\frac12m(m+3)$. This again leaves $\frac12m(m-1)$ factors to be determined.

Using exactly the same row operations as in the previous case shows that the missing factors are exactly the same. Therefore, $\det K$ is proportional to $\theta^m(\theta-1)^m(\theta+1)^{m-1}(\theta-2)^{m-1}\dots(\theta+m-1)(\theta-m)$. To determine the constant of proportionality, note that if $\theta=m+1$ then $\det K=1$. This gives eq.~\eqref{detKodd}.

Again, this is a special case  of Theorem 1 of \cite{RHPK} (Thm.~\ref{RHPKthm}).  It corresponds to $N=m$  and $L=\{m+2,m+3,\dots,2m+1\}$.
\end{proof}

In either case, $\det K$  is manifestly never zero for $\theta \in (0,1)$.

\begin{prop}
\label{DetB}
For the matrix $B$ of \eqref{Ba}, \eqref{Bb}, 
\begin{equation}
\label{det2}
\det B = \theta^{\lfloor\frac{n}2\rfloor}\prod_{k=1}^{n-2} \frac{\left(\theta^2-k^2\right)^{\lfloor\frac{n-k}2\rfloor}}{(2k+1)!!}.
\end{equation}
which clearly does not vanish for $\theta \in (0,1)$.
\end{prop}

One of us (RRH) and Philip Keningley will show in a forthcoming paper \cite{RHPK} that one can  evaluate a more general class of determinants, including $\det B$.   The following statement is a paraphrase of Theorem 1 of \cite{RHPK}.
\begin{thm}
\label{RHPKthm}
{\bf (R.R.~Hall and P.~Keningley \cite{RHPK})} Let $t$ be a real variable and let $L = \{l_1, l_2, \dots l_N\}$  be a set of integers such that $1 \le l_1 < l_2 <
\dots < l_N$.   Let $A$ be the $N\times N$ matrix with components
\[
A_{ij}=\binom{t}{l_i-j}.
\]
Then $\det A$ is given by the formula
\beq
\label{RHPK}
\det A = \kappa(L) \prod_{h=0}^{l_N - N}(t-h)^{\alpha_h} \prod_{k=1}^{N-1} (t+k)^{\beta_k}
\eeq
where $\alpha_h = \operatorname{card}\{i \mid l_i > N+h\}$ and $\beta_k = \operatorname{card}\{i \mid i >k, l_{i-k} > N-k\}$
(so $\sum_{h=0}^{l_N - N}\alpha_h + \sum_{k=1}^{N-1}\beta_k = l_1+l_2 + \dots + l_N - \binom{N}{2}$) and 
\[
\kappa(L)=\frac{\prod_{1\leq i<j\leq N}(l_j-l_i)}{\prod_{k=1}^N (l_k - 1)!}.
\]
\end{thm}

\begin{proof}[Proof of Prop.~\ref{DetB}]
To see how \eqref{det2} arises as a special case of Theorem \ref{RHPKthm}, first observe that if we set $N=n$, $t=-\theta$ and $l_i=2i-1$, then 
\[
B_{ij}= b_{2i-j-1} = (-1)^{j+1}\binom{t}{2i-j-1}= (-1)^{j+1} A_{ij} ,
\]
so $\det B = (-1)^{\lfloor\frac{n}2\rfloor}\det A$. Working through the definitions in that theorem gives,
\[
\alpha_h = \card\{i\leq n \mid 2i-1>n+ h\} = \lfloor\tfrac{n-h}2\rfloor
\]
and
\[
\beta_k = \card\{i\leq n \mid i>k,\ 2i-2k-1> n-k\} = \alpha_k = \lfloor\tfrac{n-h}2\rfloor .
\]
Note that $\alpha_{n-1}=0$. Equation \eqref{RHPK} gives,
\begin{align*}
\det A &= \kappa(L) \prod_{h=0}^{n-1} (t-h)^{\alpha_k} \prod_{k=1}^{n-1} (t+k)^{\beta_k} 
= \kappa(L) t^{\lfloor\frac{n}2\rfloor} \prod_{k=1}^{n-2} (t^2-k^2)^{\lfloor\frac{n-k}2\rfloor} \\
&= (-1)^{\lfloor\frac{n}2\rfloor}\kappa(L) \theta^{\lfloor\frac{n}2\rfloor} \prod_{k=1}^{n-2} (\theta^2-k^2)^{\lfloor\frac{n-k}2\rfloor} ,
\end{align*} 
Thus
\beq
\label{almost B}
\det B = \kappa(L) \theta^{\lfloor\frac{n}2\rfloor} \prod_{k=1}^{n-2} (\theta^2-k^2)^{\lfloor\frac{n-k}2\rfloor} .
\eeq

Now,
\[
\kappa(L) = \frac{\prod_{1\leq i<j\leq n} 2(j-i)}{\prod_{k=1}^n (2k-2)!} .
\]
The numerator can be written as
\[
\prod_{k=0}^{n-2} 2^{k+1}(k+1)!
\]
and the denominator (after shifting $k$) as
\[
\prod_{k=0}^{n-2} (2k+2)! .
\]
This gives
\begin{align*}
\kappa(L)&= \prod_{k=0}^{n-2} \frac{2^{k+1}(k+1)!}{(2k+2)!} = \prod_{k=1}^{n-2} \frac{2^{k+1}(k+1)!}{(2k+2)!} 
= \prod_{k=1}^{n-2} \frac{2^{k}k!\cdot 2(k+1)}{(2k+1)!\cdot(2k+2)} \\
&= \prod_{k=1}^{n-2} \frac{2^{k}k!}{(2k+1)!} = \prod_{k=1}^{n-2} \frac{1}{(2k+1)!!} .
\end{align*}
Combining this with eq.~\eqref{almost B} gives eq.~\eqref{det2}.
\end{proof}

To give a more self-contained proof of Proposition~\ref{DetB}, we will rely on the following lemma.
\begin{lem} 
\label{Det2Lemma}
For $m\in \mathbb{N}$, if $A_m$ is the $m\times m$ matrix whose $ab$'th element is $\binom{m+1}{2a-b}$, i.e.,     
\[
A_1=2, \quad A_2 = \begin{pmatrix}3 & 1 \\ 1 & 3\end{pmatrix}, 
\quad A_3 = \begin{pmatrix}4 & 1 & 0 \\ 4 & 6 & 4 \\ 0 & 1 & 4\end{pmatrix} \quad \hbox{etc.,}
\]
then $\det A_m = 2^{m(m+1)/2}$
\end{lem}

\begin{proof}
Let $V_m$ be the $m\times m$ lower triangular matrix for which 
\[
({V_m})_{ab} =
\begin{cases}
1 & a\geq b\\
0 & a<b
\end{cases}
\]
and note that $V_m^{-1}$ is the matrix 
\[
(V_m^{-1})_{ab} =
\begin{cases}
1 & a=b\\
-1 & a=b+1\\
0 & \text{otherwise}.
\end{cases}
\]
 Then we claim that, for $m \in \mathbb{N}$, $m>1$,
\begin{equation}
\label{conjug}
V_mA_mV_m^{-1} =\begin{pmatrix} A_{m-1} & {\bf v} \\ {\bf 0} & 2^m\end{pmatrix}
\end{equation}
where $\bf 0$ is the $1 \times (m-1)$ row vector $\begin{pmatrix}0 & 0 & \dots & 0\end{pmatrix}$ and $\bf v$ is some $(m-1) \times 1$ column vector the values of whose elements we do not need to consider.  For example, 
\[
\begin{pmatrix} 1 & 0 & 0 \\ 1 & 1 & 0 \\ 1 & 1 & 1\end{pmatrix}\begin{pmatrix} 4 & 1 & 0\\ 4 & 6 & 4\\ 0 & 1 & 4\end{pmatrix}\begin{pmatrix}1 & 0 & 0\\-1 & 1 & 0\\0 & -1 & 1\end{pmatrix} = \begin{pmatrix} 3 & 1 & 0\\1 & 3 & 4\\0 & 0 & 8\end{pmatrix} .
\]
(So in this case, $\bf v$ turns out to be $\left(\begin{smallmatrix}0\\4\end{smallmatrix}\right)$.)
Clearly it follows from \eqref{conjug} that, for $m>1$,
\[
\det A_{m} = 2^m\det A_{m-1}
\]
while we know $\det A_1 =2$.  Hence, by induction, $\det A_{m} = 2^1\cdot2^2\cdots 2^m = 2^{m(m+1)/2}$.

To prove the above claim, we notice that \eqref{conjug} amounts, in index notation, to the following pair of sets of equations:  First, for  any $m\ge 1$, $i \le m$, $j<m$,
\begin{equation}
\label{ilemjm}
\sum_{p=1}^i\binom{m+1}{2p-j}-\binom{m+1}{2p-j-1}=\binom{m}{2i-j}
\end{equation}
where, in particular, we notice that for $i=m$ and $j<m$, the right hand side is zero; and second, for any $m \ge 1$,
\[
\sum_{p=1}^m\binom{m+1}{2p-m}=2^m.
\]
It is not difficult to see that these sets of equations follow from the well-known binomial identities (see e.g.\  \url{https://proofwiki.org/wiki/Properties_of_Binomial_Coefficients})
\[
\sum_{k=0}^n (-1)^k\binom{r}{k}=(-1)^n\binom{r-1}{n}
\]
\[
\sum_{i\ge 0}\binom{n}{2i}=2^{n-1}; \quad \sum_{i\ge 0}\binom{n}{2i+1}=2^{n-1}.
\]
For example, for $r\ge 5$, the first of these identities in the case $n=3$ entails equation \eqref{ilemjm} both when $m=r-1$, $i=2$ and $j=1$ and also when $m=r-1$, $i=3$ and $j=3$.
\end{proof} 

\begin{proof}[Alternative proof of Prop.~\ref{DetB}]
This relies on showing first that $\det B$ is either even or odd. (In fact it is even for $\lfloor\frac{n}2\rfloor$ even and odd for $\lfloor\frac{n}2\rfloor$ odd).   Once this is established, we may argue as follows:  Since all elements of a row of $B$ which ends in $b_m$ have $b_m$ as a common factor, $\det B$ must be divisible by $b_1b_3\dots b_{n-1}$ when $n$ is even and by
$b_2b_4\dots b_{n-1}$ when $n$ is odd.   I.e.\ it must be divisible by $\theta^{\lfloor\frac{n}2\rfloor}\prod_{k=1}^{n-2} \left(\theta + k\right)^{\lfloor\frac{n-k}2\rfloor}$ -- which is an order $\left (n(n-1)/2+{\lfloor\frac{n}2\rfloor}\right)/2$ polynomial in $\theta$.   On the other hand,  $\det B$ is clearly an order $n(n-1)/2$ polynomial.  The only way this can be true and $\det B$ be even or odd is if it takes the form \eqref{det2} up to an overall constant. For example, when $n=6$, we know $\det B$ is divisible by $\theta^3(\theta+1)^2(\theta+2)^2(\theta+3)(\theta+4)$ while $\det B$ is an order 15 polynomial in $\theta$ and also odd.  Clearly, the only way this can be true is if $\det B$ is $\theta^3(\theta+1)^2(\theta+2)^2(\theta+3)(\theta+4)$ times a constant times $(\theta-1)^2(\theta-2)^2(\theta-3)(\theta-4)$.

To prove that $\det B$ is either even or odd, define $D$ similarly to $B$ except that $b_m$ is replaced by $d_m$ where
\begin{equation*}
d_m= \frac{\theta(\theta -1) \dots (\theta - m+1)}{m!}
\end{equation*}
Then it clearly suffices to show that $\det B=\det D$.    To show this, we first
define $\tilde B$ and $\tilde D$ to be the $(n-1)\times (n-1)$ matrices obtained by striking out the leftmost column and the topmost row of each of $B$ and $D$.  Clearly $\det\tilde B=\det B$ and 
$\det\tilde D=\det D$.   We will exhibit an $(n-1)\times (n-1)$ lower triangular matrix, $W$, whose diagonal elements are all 1, such that 
\begin{equation}
\label{triang}
\tilde B = W\tilde D 
\end{equation}
from which we can immediately conclude that $\det B=\det\tilde B=\det\tilde D=\det D$.

We define $W$ to have components
\begin{equation*}
W_{ab} = (-1)^{a-c}\binom{-\theta}{a-c}.
\end{equation*}
(Here, and below, we use the usual conventions that, for any $v$, when $m=0$, $\binom{v}{m}=1$ and, when $m<0$, $\binom{v}{m}=0$.)
Equation \eqref{triang} then amounts, in index notation, to the equation
\begin{equation}
\label{matreq}
(-1)^c\binom{-\theta}{2a-c}=\sum_{b=1}^a (-1)^{a-b}\binom{-\theta}{a-b}\binom{\theta}{2b-c}.
\end{equation}
This follows immediately from the binomial identity
\begin{equation}
\label{binomident}
(-1)^q\binom{-t}{q}=\sum_{(r,s)\in \{\mathbb{N}_0\times \mathbb{N}_0| 2r+s=q\}} (-1)^r\binom{-t}{r}\binom{t}{s}
\end{equation}
which is, in turn, easily demonstrated by equating coefficients of powers of $t$ after binomially expanding each of the terms in the equation 
\begin{equation*}
(1-x)^{-t}=(1-x^2)^{-t}(1+x)^{t}.
\end{equation*}
To see that \eqref{matreq} follows from \eqref{binomident}, we take $q$ to equal $2a-c$, $r$ to equal $a-b$, and $s$ to equal $2b-c$.   In particular, it is easy to see that the sum from $b=1$ to $a$ in \eqref{matreq} then amounts to the sum over the set  $\{(r,s)\in \mathbb{N}_0\times \mathbb{N}_0\mid 2r+s=q\}$.  For example, for any $n \ge 4$ (so that $\tilde B$ and $\tilde D$ are at least $3\times 3$ matrices) if $a=3$ and $c=2$, then $q=4$ and, as $b$ increases from 1 to 3, $(r,s)$ ranges over the set $\{(2,0), (1,2), (0,4)\}$.

To fix the overall constant in \eqref{det2}, we notice that, for $\theta=n-1$, 
$\tilde D_{ab}$ is the $(n-1) \times (n-1)$ matrix with $ab$ component $\binom{n-1}{2a-b}$.   Moreover this has bottom row 
$\begin{pmatrix} 0 & 0 & \dots & 0 & 1\end{pmatrix}$ and so, by expanding along its bottom row, its determinant is the same as the determinant of the $(n-2) \times (n-2)$ matrix with the same formula for its $ab$'th element.   By Lemma \ref{Det2Lemma}, this has  determinant $2^{(n-1)(n-2)/2}$. 
Thus 
\begin{equation*}
\det B = \frac{2^{(n-1)(n-2)/2}\theta^{\lfloor\frac{n}2\rfloor}\prod_{k=1}^{n-2}\left(\theta^2-k^2\right)^{\lfloor\frac{n-k}2\rfloor}}{(n-1)^{\lfloor\frac{n}2\rfloor}\prod_{k=1}^{n-2}\left((n-1)^2-k^2\right)^{\lfloor\frac{n-k}2\rfloor}}
\end{equation*}
which is easily seen to be an alternative way of expressing eq.~\eqref{det2}.
\end{proof}

\end{document}